\documentclass[11pt]{article}
\usepackage[top=2cm,bottom=3cm,left=2.5cm,right=2.5cm]{geometry}
\usepackage[utf8]{inputenc}
\usepackage[english]{babel}
\usepackage{authblk} 
\usepackage{amsmath}
\usepackage{cases}
\usepackage{amsthm}
\usepackage{amssymb}
\usepackage{hyperref} 

\usepackage{amsfonts}

\usepackage{verbatim} 
\usepackage{graphicx}
\usepackage[dvipsnames]{xcolor}
\usepackage{hyperref}
\RequirePackage{doi}
\usepackage{accents}
\usepackage{float} 
\usepackage{caption}
\hypersetup{colorlinks=true,urlcolor=blue,linkcolor=black,citecolor=black}
\graphicspath{{figures/}}

\usepackage{color,colortbl}  

\newtheorem{Theorem}{Theorem}[section] 
\newtheorem{Lemma}{Lemma}[section]
\newtheorem{Corollary}{Corollary}[section]
\newtheorem{Proposition}{Proposition}[section]

\newtheorem{Remark}{Remark}[section]
\newtheorem{Conjecture}{Conjecture}[section] 

\numberwithin{equation}{section}

\definecolor{azure}{rgb}{0.94, 1.0, 1.0}

\definecolor{Gray}{gray}{0.90}

\newcommand{\utilde}[1]{\underaccent{\tilde}{#1}}

\newcommand{\R}{{\mathbb R}}
\newcommand{\N}{{\mathbb N}}

\newcommand{\EE}{{\mathbb E}}

\newcommand{\E}{{\mathbb{E}}}

\newcommand{\sT}{{\mathsf T}}
\newcommand{\sX}{{\mathsf X}}
\newcommand{\sY}{{\mathsf Y}}

\newcommand{\usX}{{\underline{\mathsf X}}}

\renewcommand{\ldots}{ ... }

\newcommand{\um}{\underline m}

\newcommand{\utpi}{\utilde \pi}
\newcommand{\otpi}{\tilde \pi}

\newcommand{\oD}{L}
\newcommand{\upi}{\underline{\pi}}

\newcommand{\sfrac}[2]{{\textstyle\frac{#1}{#2}}}

\definecolor{LightCyan}{rgb}{0.88,1,1}

\newcommand{\nn}{\nonumber}
\newcommand{\one}{{\mathbf 1}\hskip-.5mm}

\DeclareMathOperator*\argmin{arg \, min}

\DeclareMathOperator{\var}{Var}
\DeclareMathOperator{\cov}{Cov}

\newcommand{\Dopt}{D^{\rm opt}}

\newcommand{\alogq}{|\log(q)|}

\def\quotient#1#2{%
    \raise1ex\hbox{$#1$}/\lower1ex\hbox{$#2$}%
}
\title{Group testing with nested pools}

\author{In\'es Armend\'ariz\thanks{Universidad de Buenos Aires \& IMAS-CONICET-UBA. Email: {\tt iarmend@dm.uba.ar}},\;
    Pablo A. Ferrari\thanks{Universidad de Buenos Aires \& IMAS-CONICET-UBA. Email: {\tt pferrari@dm.uba.ar}},
    Daniel Fraiman\thanks{Universidad de San Andr\'es \& CONICET. Email: {\tt dfraiman@udesa.edu.ar}},\\
    Jos\'e Mario Mart\'{\i}nez\thanks{Universidade de Campinas. Email: {\tt martinez@ime.unicamp.br}},\;
    Silvina Ponce Dawson\thanks{Universidad de Buenos Aires \& IFIBA-CONICET-UBA Email: {\tt silvina@df.uba.ar}}}

\begin{document}
{\tiny
  \maketitle
}

\begin{abstract}
In order to identify the infected individuals of a population, their samples are divided in equally sized groups called pools and a single laboratory test is applied to each pool. Individuals whose samples belong to pools that test negative are declared healthy, while each pool that tests positive is divided into smaller, equally sized pools which are tested in the next stage.  In  the $(k+1)$-th stage all remaining samples are tested. 
If $p<1-3^{-1/3}$, we
minimize the expected 
number of tests per individual as a function of the number $k+1$ of stages, and of the pool sizes in the first $k$ stages.
We show that for each $p\in (0, 1-3^{-1/3})$ the optimal choice is one of four possible schemes, which are explicitly described. We conjecture that for each $p$, the optimal choice is one of the two sequences of pool sizes $(3^k\text{ or }3^{k-1}4,3^{k-1},\dots,3^2,3 )$, with a precise description of the range of $p$'s where each is optimal. The conjecture is supported by overwhelming numerical evidence for $p>2^{-51}$.
We also show that the cost of the best among the schemes $(3^k,\dots,3)$  is of order $O\big(p\log(1/p)\big)$, comparable to the information theoretical lower bound $p\log_2(1/p)+(1-p)\log_2(1/(1-p))$, the entropy of a Bernoulli$(p)$ random variable.

\vspace{3mm}
\noindent{\em Keywords } Dorfman's retesting, Group testing, Nested pool testing, Adaptive pool testing.\\[1mm]
{\em AMS Math Classification } Primary 62P10

\end{abstract}

  \section{Introduction} 
  \label{pblm}
The purpose of group testing is to identify the set of infected individuals in a large population in the most efficient way.
Dorfman \cite{dorfman} was the first to propose a group testing strategy in 1943. The samples from $n$ individuals are pooled and tested together. If the test is negative, the $n$  individuals of the pool are cleared. If the test is positive, each of the $n$ individuals must be tested separately, and $n+1$ tests are required  to test $n$ people. In the present note we study a sequential multi-stage extension of Dorfman's algorithm, considered by Kotz and Johnson, see \cite{MR673451,MR1214790}. This scheme belongs to the family of \emph{adaptive models}, in which the course of action chosen at each stage depends on the results of previous stages.

\paragraph{\sl Model and results} Consider a population size $N$ and a collection of independent random variables $(\sX_i)_{i=1}^N$ with common Bernoulli distribution with a known parameter $p$. Each variable represents the infection status of an individual in the population, and $p$ is the prevalence. The goal is to reveal the values of a realization of the variables by evaluating a series of suitably chosen functions of groups of the variables (tests), while minimizing the average number of tests per variable. 
We study the following adaptive scheme.
Let $k\in \N$ and consider a sequence of integers $m_1>\dots>m_k>1$, where $m_j $ is a multiple of $m_{j+1}$ for all $j$.
In the first stage the variables are pooled into groups (pools) of size $m_1$, and a single test is performed per group.
The test has two possible outcomes, positive of negative. A positive result indicates that at least one of the variables in the pool takes the value $1$. Variables in pools that test positive are split into sub-pools of size $m_2$, which are tested in the second stage. Continue in this way until the $k$-th stage, when the yet undetermined variables are grouped into pools of size $m_k$. In the $(k+1)$-th stage, evaluate all variables belonging to pools that tested positive in the $k$-th stage. This procedure was previously considered by Kotz and Johnson \cite{MR1214790}, see also \cite{MR673451}. 
The scheme is \emph{nested} because pools in each stage are subsets of the pools in the previous stage.
For each infection probability $p$ and strategy $(k,\um)$ with $\um=(m_1,\dots,m_k)$, let 
\begin{align*}
\sT_k(\um)&:= \text{total number of tests per initial pool (of size $m_1$)},\\
D_k(\um,p)&:= \frac{1}{m_1} \E \sT_k(\um)=\text{cost of the strategy when the prevalence is $p$,}
\end{align*}
where $\E$ denotes mathematical expectation.
In Section \ref{dependents} we describe $\sT_k(\um,p)$ as a function of $(\sX_1,\dots,\sX_{m_1})$, and compute its expectation and variance. Since the variables are iid, the cost associated to every pool in the first stage is the same, and $D_k(\um,p)$ computes the cost of the scheme if the population size $N$ is a multiple of $m_1$. If this is not the case, there will be a remainder pool of size less than $m_1$. Variables in this group are individually tested, leading to at most $m_1$ extra tests, which produces an error in the average number of tests per variable that can be bounded by $m_1/N$. Since $m_1$ is independent of $N$, and this is typically large, in the sequel we will disregard this error, assuming that either $N=\infty$ or $N$ is a multiple of $m_1$, see also the Remark after Proposition \ref{stk1}.

The optimal strategy associated to $p$ is the choice $(k,\um)$ that minimizes $D_k(\um,p)$. We solve this optimization problem in Section \ref{optimization}. If $p\ge 1-3^{-1/3}$, the best strategy is to test each individual sample. For $0<p<1-3^{-1/3}$, we prove in Proposition~\ref{p12} that the optimal strategy is $(k,\um)= \bigl(1,(3)\bigr)$ or $\bigl(1,(4)\bigr)$, or $k\ge2$ and $(k,\um)$ being one of the following four types
\begin{align}
k&=k_{34}(p),&\um_{34}(p)&:=( 3^{k-1}4, \,3^{k-1},\dots,\,3^2, 3),\label{i1}\\[1.5mm]
 k&=k_3(p),&\um_{33}(p)&:=(3^k,3^{k-1},\dots,\,3^2,\,3),\label{i2}\\[1.5mm]
 k&=k_{23}(p),&\um_{23}(p)&:=(3^{k-1}2,\,3^{k-2}2,\dots,\, \,3^2 2,\,3{\textstyle \times} 2,\,2), \label{i3}\\[1.5mm]
 k&=k_{24}(p),&\um_{24}(p)&:=(4\textstyle{\times}3^{k-2}\textstyle{\times} 2,\, 3^{k-2}2,\dots,\,3\times 2,\,2), \label{i4}
\end{align}
and give precise formulas for the optimal number $k$  for each of these cases in Proposition~\ref{p13}. In words, the size of the $k$-th pool must be $2$ or $3$, each intermediate 
pool must have size equal to $3$ times the size of the pool in the following stage, and the size of the first pool is $3$ or $4$ times the size of the second pool. For $1-3^{-1/3}>p>2^{-51}$, the strategies  $\um_{24},\um_{23}$ were  numerically compared  with $\um_{34},\um_{33}$ using multiple precision computations and careful floating point analysis, and discarded for $p$ in that range. Conjecture~\ref{c8} proposes that the optimal strategy is either $\um_{33}$ or $\um_{34}$; in this case we identify the minimizing strategy for each $p$.

In Theorem \ref{thm} we show that the cost of the optimal strategy associated to $p$ is at most $\frac{3}{\log 3} \,p\log(1/p)$, and, moreover,  it differs from the cost of the strategy $(3^{k_3(p)},\dots,3)$, $k_3(p)$ chosen as in~\eqref{i2}, by a fraction of this quantity. As a first conclusion, this proves that the cost of the optimal strategy is of the same order as the information theoretical lower bound \cite{zbMATH03222508} and the cost of the best known algorithms \cite{MR3569134,sobel1959group}, although the constants are slightly worse, $3/\log 3$ in contrast to~$1/\log 2$. A second upshot of Theorem~\ref{thm} is that the testing scheme that consists in applying the strategy $\big(k_3(p),(3^{k_3(p)},\dots, 3)\big)$ has optimal, or near optimal cost, within the family of strategies studied in this paper, for all $p \in (0,1-3^{-1/3})$. 

The optimal nested scheme proposed here requires a deterministic number of stages of order $\log_3(1/p)$ per initial pool of size $m_1$, see Proposition~\ref{p13}.

\paragraph{\sl Background}
Dorfman's $2$-stage strategy was subsequently improved
by Sterrett~\cite{sterrett} to further reduce the number of tests, and extended
to more stages of group testing by  Li~\cite{Li} and Finucan~\cite{finucan}.
Sobel and Groll~\cite{sobel} introduced the idea of recursively splitting  the initial pool. Building on this idea, the procedure of \emph{binary splitting} halves the 
size of a positive testing pool, at each subsequent stage, requiring  $\log_2(m)$ tests (one per stage) to find one infected individual in an initial pool of size $m$. Improvements on this method 
yield procedures that identify all infected samples in the initial pool at a cost of $p\log_2(1/p)$ tests per individual, see Hwang~\cite{hwang}, Allemann~\cite{Allemann} and Aldridge~\cite{aldridge}. In these methods the number of stages required to complete the exploration of the initial pool is at least linear in the random number of infected individuals. This might be a concern if each stage is lengthy.

An alternative approach is given by \emph{nonadaptive testing}, where the composition of the pools is designed in advance, and usually many tests are conducted in parallel. In the situation considered here, where the number of infected samples grows linearly in the size of the population, and the goal is to identify all infected individuals (zero error), the best nonadaptive testing strategy is to simply test each individual. On the other hand, there are intermediate algorithms combining adaptive pooling with a limited number of stages and parallel testing, that perform very well in the limit when $p\to 0$, see M\'ezard and Toninelli~\cite{mezard-toninelli}  and Damaschke and Muhammad~\cite{damas}. We  refer the reader to the recent book by Aldridge, Johnson and Scarlett \cite{review-johnson} for a thorough discussion of different group testing techniques.

Our results assume that the infection probability $p$ is known and common to all individuals, and that each test is conclusive, i.e. there are no false positives/negatives. When this is not the case, 
Sobel and Groll~\cite{sobel} extend the approach to include the estimation of  $p$, and consider the situation where there are
subpopulations
characterized by different infection probabilities. 
References ~\cite{MTB,bilder-tebbs,bilder2015} propose to use information on heterogeneous populations
to improve Dorfman's algorithm. An
extension of Dorfman's algorithm in which each group is tested several
times to minimize testing errors was presented
in~\cite{doi:10.1080/00401706.1972.10488888}. 
The impact of test sensitivity in both
adaptive/nonadaptive testing was analyzed in~\cite{kim_2007}. We also assume that the infection status of the individuals in the population are independent random variables, and do not vary during the time required to run the testing procedure. A recent work \cite{srinivasavaradhan2021dynamic} considers the case when the infection spreads during the testing period. Other works consider the situation when samples are correlated due to the presence of communities in the population \cite{ahn2021adaptive,goenka2020contact,nikolopoulos2021group,nikolopoulos2021community,bertolotti2020network,gabrys2021acdc}.

This paper was triggered by discussions with a group of biologists from Santa Fe in Argentina who intended to use massive testing as one of the tools to control the Covid-19 pandemic in the country. 
The possibility of running
the current gold standard test, RT-qPCR, in sample pools was investigated in~\cite{Yelin2020.03.26.20039438}, finding that the
identification of individuals infected with SARS-CoV-2 is in fact possible using mixtures of up to 32 individual samples.  The use of more sensitive tests~\cite{Suo2020.02.29.20029439,Dong2020.03.14.20036129} would likely improve this limit.
In~\cite{hanel2020boosting}, Dorfman's algorithm is applied with replicates that check for false negatives or
positives. In~\cite{Mentus2020.04.05.20050245}, adaptive and
non-adaptive methods that use binary splitting are
compared numerically.
The work in~\cite{Sinnott-Armstrong2020.03.27.20043968} evaluates
numerically the performance of two-dimensional array pooling comparing it
with Dorfman's strategy.

\paragraph{\sl Complementary work} The authors, in collaboration with Hugo Menzella, have done some complementary work targeting  applications. The paper \cite{affmmpd-medrxiv} discusses prevalence estimation, group testing under the presence of errors, and technical verification use of sensitive ddPCR Covid tests. The webpage \cite{pooling-webpage} provides an interface to choose the best testing strategy adapted to various possible restrictions such as prevalence, maximum initial pool size, or maximum possible number of stages, and provides information on the expected cost and processing time of different strategies; it is work in collaboration with Francisco Sobral.  

\paragraph{\sl Organization}

The article is organized as follows. In Section \ref{dependents} we define the nested procedure and compute the expectation and variance of the random number of tests per individual. 
We find the optimal strategy in Section \ref{optimization}. In Section \ref{optim} we study a linearization of the cost. We apply some of the conclusions in Section \ref{feas} to show that the strategy $m_{33}$ is optimal or close to optimal for all $p$. We include an appendix with technical computations.

\section{Nested strategies}
\label{dependents}

In this section we define multi-stage nested strategies and obtain a formula for the expected number of tests per individual, which we will call the \emph{cost} of the strategy.

Let $p$ be the probability that an individual is infected. The cost of Dorfman's \cite{dorfman} 2-stage strategy, with groups of $m$ individuals in the first stage and individual testing in the second stage, is given by
\begin{align}
\label{meandorf}
\frac{1}{m}\bigl(1+ m(1-(1-p)^m)\bigr).
\end{align}
The term 1 in the parenthesis is the number of tests in the first stage: one per pool. The second term is the expected number of tests in the second stage: 0 times the probability $(1-p)^m$ that the test of the first stage is negative, plus $m$ (additional tests) times $1-(1-p)^m$, the probability that the test in the first stage is positive. 
Since the cost of testing all individuals is $1$,  Dorfman's strategy is worth pursuing only if the cost \eqref{meandorf} is less than 1, solving for $p$ we get 
\begin{align}
\label{pdorf}
p<\max_{m=2,3,\dots}\bigl( 1-m^{-1/m}\bigr) =  1-3^{-1/3 }\approx 0.3066,
\end{align}
and in this case $n$ must be greater than or equal to $3$. Ungar \cite{ungar1960cutoff} showed that for $p> (3-\sqrt 5)/2= 0.3819\dots$, individual testing is the best strategy, and recently Papanicolaou \cite{papanicolaou2020binary} provided a strategy that improves individual testing for $p$ below $(3-\sqrt 5)/2$.

We iterate Dorfman's strategy using nested pools, that is, pools in each stage are obtained as a partition of the positive-testing pools in the previous stage.

Let $\usX=(\sX_1,\dots,\sX_N)$ be a family of independent, identically distributed random variables $\sX_i\sim$ Bernoulli$(p)$, where $\sX_i=1$ means that the individual $i$ is infected. For any subset $A\subset\{1,\dots,N\}$, the function
  \begin{align}
     \phi_A(\usX) := \prod_{i\in A} (1-\sX_i).
  \end{align}
is called a \emph{test}.  The test takes the value $1$ if $A$ has no infected individuals. Otherwise it vanishes and in this case we say that the pool $A$ is infected.
The goal is to reveal the values of $\sX_i$ for all $i$ as the result of a family of tests.

As an example, we describe now the computation of the cost for the $3$-stage procedure, the general case is summarized in \eqref{Tdepk} and \eqref{meankdep}. 
Denote by $m_1,m_2\ge 1$ the sizes of the pools in the first and second stages, respectively, where $m_1$ is a multiple of $m_2$. Let
 $A_1$ be a pool in the first stage, and partition $A_1$ into $\frac{m_1}{m_2}$ subsets
$A_{1,1},\dots, A_{1, \frac {m_1} {m_2}}$ of size $m_2$ each in the second stage. 
Let 
\begin{align}
\label{phi-again}
\sY_1:=\phi_{A_1}(\usX)\quad\text{and}\quad \sY_{1,i}:=\phi_{A_{1,i}}(\usX),\qquad 1\le i\le \frac{m_1}{m_2}. 
\end{align}
The first stage is to evaluate $\sY_1$. If $\sY_1=1$, then there are no infected individuals in $A_1$. 
If $\sY_1=0$ then go to the second stage and evaluate $\sY_{1,i},  1\le i\le \frac{m_1}{m_2}$. 
Finally, in the third stage, apply the test to each individual belonging to an infected pool of the second stage.
Denote by $T_2$ the total number of tests under this $3$-stage scheme. We get
\begin{align}
\label{Tdep}
\sT_2=1+(1-\sY_1)\frac{m_1}{m_2}+(1-\sY_1)\,m_2\sum_{i=1}^{\frac{m_1}{m_2}}(1-\sY_{1,i}).
\end{align}
The first term $1$ is the initial test of $A_1$. The second term counts the tests of $A_{1,i}$, $i=1,\dots,\frac{m_1}{m_2}$. These tests are performed only when $(1-\sY_1)=1$. The third term counts the individual tests: $m_2$ for each $A_{1,i}$ with  $(1-\sY_1)(1-\sY_{1,i})=1$.
Since $\sY_1=1$ implies $\sY_{1,i}=1$ for all $i$, we have
\begin{align}
\label{cancel}
(1-\sY_1)(1-\sY_{1,i})= (1-\sY_{1,i}),
\end{align}
which in turn implies
\begin{align}
\label{Tdep2}
\sT_2=1+(1-\sY_1)\frac{m_1}{m_2}\,+\,m_2\sum_{i=1}^{\frac{m_1}{m_2}}(1-\sY_{1,i}).
\end{align}
It is useful to denote
\begin{align}
  \label{defq}
  q:=1-p.
\end{align}
Since $1-\sY_1$ is Bernoulli$(1-q^{m_1})$ and $(1-\sY_{1,i}:{1\le i\le \frac{m_1}{m_2}})$ are independent, identically distributed  Bernoulli$\big(1-q^{m_2}\big)$ random variables, we get that the \emph{cost} of the strategy $(k,\um)=\bigl(2, (m_1,m_2)\bigr)$ is
\begin{align}
\label{mean3dep}
D_2(\um,p) &:=\frac1{m_1}\, \EE\sT_2 =\frac1{m_1} + \frac1m_2 \big(1-q^{m_1}\big)+1-q^{m_2},
\end{align}
where we used \eqref{Tdep2} and $\EE$ denotes mathematical expectation.

Denote the set of nested strategies with $k$ pooled stages and the set of nested strategies by
\begin{align}
\label{cN}
{\cal M}_k&:=\big\{\big(k,(m_1,\dots, m_k)\bigr): m_j>m_{j+1},\; m_j \text{ is a multiple of } m_{j+1},\; j=1,\dots,k-1\big\}
\\[1mm]
{\cal M}&:=\cup_{k\ge 0}{\cal M}_k.\label{cN1}
\end{align}

Given a nested strategy $(k,\um)$, label the pools by a rooted tree with vertices $\{1\}\cup\bigl\{(1,i_2,\dots, i_j): 1\le i_j\le m_{j-1}/m_j, 2\le j\le k\bigr\}$ and edges $[(1),(1,i_2)]$,  $[(1,i_2),(1,i_2,i_3)],\dots$ connecting a pool of stage $j-1$ with all its sub-pools of stage $j$, see Fig.~\ref{fig:tree-27}.  With this notation we have
\begin{align}
  \bigcup_{i_j=1}^{m_{j-1}/m_{j}}A_{1,i_2,\dots,i_j} = A_{1,i_2,\dots,i_{j-1}},\quad j=2,\dots, k.
\end{align}
\noindent
\begin{minipage}{\linewidth}
\begin{center}	
	 \includegraphics[width=.7 \textwidth, trim={0 0mm 0 15mm},clip]
             {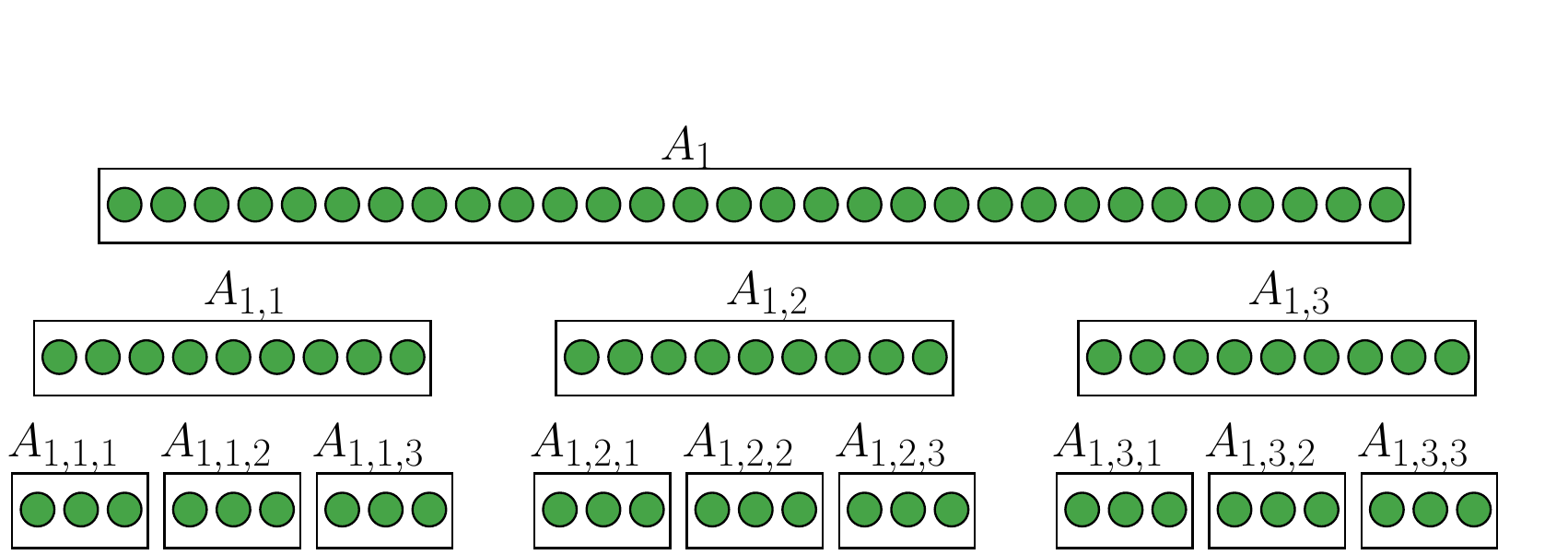}
	\captionof{figure} {Labeled pools for the strategy $k=3$ and $\um=(27,9,3)$. \label{fig:tree-27}	 }
      \end{center}
    \end{minipage}\\
    
 The following proposition was obtained in \cite{MR673451}, see also \cite{MR1214790}. We include the proof for completeness.
 \begin{Proposition}
   \label{stk1}
      Denote by $\sT_{k}$ the number of tests  needed to evaluate all individuals in $A_1$ with the strategy $(k,\um)$. Then,
\begin{align}
\label{Tdepk}
  \sT_{k}&=1+\frac{m_1}{m_2}(1-\sY_1)
         +\frac{m_2}{m_3} \sum_{i_2=1}^{\frac{m_1}{m_2}}(1-\sY_{1,i_2})
         +\frac{m_3}{m_4}\sum_{i_2=1}^{\frac{m_1}{m_2}}\,\, \sum_{i_3=1}^{\frac{m_2}{m_3}}(1-\sY_{1,i_2, i_3})
         \notag \\
       &\hspace{1.2cm}
        +\dots +m_k\sum_{i_2=1}^{\frac{m_1}{m_2}}\,\, \sum_{i_3=1}^{\frac{m_2}{m_3}} \dots \sum_{i_{k}=1}^{\frac{m_{k-1}}{m_k}}
         \bigl(1-\sY_{1,i_2, \dots, i_{k}}\bigr).
\end{align}
Furthermore, the cost of the strategy $(k,\um)$ is
\begin{align}
\label{meankdep}
D_{k}(\um,p)=\frac1{m_1} \EE\sT_{k}
&=\frac1{m_1} + 1-q^{m_k}+ \sum_{j=2}^{k} \frac1m_j \bigl(1-q^{m_{j-1}}\bigr).
\end{align}
\end{Proposition}
\begin{proof}
For any choice of $i_2,\dots,i_j$, with $1\le i_\ell\le m_{\ell-1}/m_\ell,\, 2\le \ell\le j$, we have
 \begin{align*}
\one_{\{\text{the pool } A_{1,i_2,\dots,i_j} \text{ tests positive}\}}=\big(1-Y_1\big) \big(1-Y_{1,i_2}\big)\dots \big(1-Y_{1,i_2,\dots,i_\ell}\big) =1-Y_{1,i_2,\dots,i_j},
 \end{align*}
 where the second identity follows by an argument similar to the one applied in \eqref{cancel}.
Once the pool $A_{1,i_2, \dots, i_{j}}$ is tested and it tests positive, its $\frac{m_j}{m_{j+1}}$ sub-pools must be evaluated. These evaluations are accounted for by the term
\begin{align*}
\frac{m_j}{m_{j+1}} \big(1-Y_{1,i_2,\dots,i_j}\big).
\end{align*}
The sums in \eqref{Tdepk} collect the contribution of all pools that tested positive to the final number of tests. 

Since $(1-\sY_{1, i_2, \dots, i_{j}})$ are i.i.d.~Bernoulli$\bigl(1-q^{m_j}\bigr)$, \eqref{meankdep} follows from \eqref{Tdepk}.
\end{proof}

\paragraph{\sl Remark}
     Strictly speaking $N$ should be a multiple of $m_1$ but in practical terms this 
  requirement is almost irrelevant. Suppose, for example, that the population $N$ is 
  $10^6$ and the prevalence $p$ is $0.001$. In this case, the optimization procedure 
   yields $m_1 = 729$ with a unitary cost equal to $0.018$. 
Therefore, $N$ is not a multiple of $m_1$. The greatest multiple of 
  $729$ that is smaller than $N$ is $N_{max} =999459$. So, the application of our strategy to 
   the set of $N_{max}$ individuals has a cost equal to $999459 \times 0.018 \approx 17990$. However, 
  in this way  $541$ individuals would remain untested. The solution to this 
  inconvenient is to include these $541$ individuals in an additional pool of $729$ individuals, 
of which $541$ are real ones and the remaining $188$ are ``slack'' individuals.   
   Thus, the total cost will be increased from 
  $0.018 \times 999459$  to $0.018 \times 1000188 \approx 18003$, which represents a negligible increase
  with respect to $17990$. 
  
\paragraph{\sl Remark} Denote $\sT_k^\ell$ the number of tests performed in the $\ell$-th stage. By the argument of Proposition \ref{stk1},
\begin{align}
\label{tkj}
\sT_k^\ell=\frac{m_{\ell-1}}{m_\ell}
\sum_{i_2,\dots,i_{\ell-1}}\bigl(1-\sY_{1, i_2, \dots, i_{\ell-1}}\bigr),
\end{align}
and its expectation does not depend on $k$: 
\begin{align}
\label{meantkj}
\EE\sT_k^\ell=\frac{m_1}{m_\ell}\big(1-q^{m_{\ell-1}}\big).
\end{align}

\paragraph{Variance}
The variance of the number of tests to be performed is useful to understand how much variability in the cost can be expected when the pooling procedure is carried out. This issue has practical importance for planning.
The variance of $\sT_k$ can be explicitly computed. We write down here the case $k=2$; the proof is given in the Appendix \ref{variance},
\begin{align}
\label{Var2}
\var\bigl( \sT_2\bigr)&=\frac{m_1^2}{m^2_2}\, q^{m_1}\bigl(1-q^{m_1}\bigr)+ m_2m_1\, q^{m_2}\bigl(1-q^{m_2}\bigr)+2\,\frac{m_1^2}{m_2}\,q^{m_1}\bigl(1-q^{m_2}\bigr).
\end{align}
An important case is when the ratio between consecutive pool sizes is constant and given by the last pool size.
Assuming $\frac{m_j}{m_{j+1}}=\mu=m_k$, $j\le 1\le k$, so that $m_j=\mu^{k-j+1}$, the variance is given by
\begin{align}
\var\bigl(\sT_{k}\bigr)&=\mu^2\bigg\{\sum_{i=1}^{k} \mu^{i-1}\bigl(1-q^{m_i}\bigr)\Big[ q^{m_i}+2\sum_{j=1}^{i-1}q^{m_j}\Big]\bigg\} \label{vark1}\\
&=\mu^2\Bigl\{\sum_{i=1}^{k} \mu^{i-1}\Bigl(1-q^{\mu^{k-i+1}}\Bigr)\Bigl[ q^{\mu^{k-i+1}}+2\sum_{j=1}^{i-1}q^{\mu^{k-j+1}}\Bigr]\Bigr\}. \label{vark2}
\end{align} 
We prove \eqref{vark1} and \eqref{vark2} in Appendix~\ref{variance}. Computations are similar for the general case, without assumptions on the sequence of pool sizes.

\section{Optimization}
\label{optimization}

Recall that ${\cal M}$ is the set of nested strategies defined in  \eqref{cN1}, and let
 \begin{align}
   \label{optn}
 \Dopt(p) := \underset{(k,\um) \in {\cal M}}{\inf}\, D_k(\um,p).
 \end{align}
We will say that $(k,\um)$ is \emph{optimal for $p$} if it is a minimizer of \eqref{optn}.

Denote
\begin{align}
  (1,\um_{33}) &:= (1,(3)); \qquad  &  (k,\um_{33}) &:= (k,(3^{k},\dots, 3)), \;\;k\ge2   \label{km23} \\
       (1,\um_{34}) &:= (1,(4));\qquad &   (k,\um_{34}) &:= (k,(4\textstyle{\times}3^{k-1},3^{k-1},\dots, 3)), \;\;k\ge2 \label{km24}\\
  (1,\um_{23}) &:=  (1,\um_{33}); \qquad  &   (k,\um_{23})   &:=(k,(3^{k-1}2,\dots,\,3^1 2,\,2)), \;\;k\ge2
\label{i22} \\
   (1,\um_{24}) &:=  (1,\um_{34});\qquad   &  (k,\um_{24})  &:=(k, (4\textstyle{\times}3^{k-2} 2,\, 3^{k-2}2,\dots,\,3^1 2,\,2)), \;\;k\ge2.
\label{i222}
\end{align}
Our next result narrows down the form of the optimal strategy for each $p$ in $[0,1]$.
\begin{Theorem}[Optimal strategies]
  \label{1of4}
  If $p\ge 1-3^{-1/3}$ the optimal strategy is to test all individuals (no pooling). If $p\le1-3^{-1/3}$, then there is a $k=k(p)\ge1$ and a strategy  $(k,\um)$ optimal for $p$ satisfying
  \begin{align}
    (k,\um) &\in \bigl\{(k,\um_{23}),(k,\um_{24}),(k,\um_{33}),(k,\um_{34})\bigr\}.
  \end{align}
\end{Theorem}
The proof of this theorem follows from a series of lemmas and Proposition~\ref{p12}.

Given a nested strategy $(k,\um)$, define $m_{k+1}:=1$ and re-write the cost $D_k(\um,p)$ given in \eqref{meankdep} as
\begin{align}
\label{ncosto4}
D_k(\um,p)=1+\sum_{i=1}^k \Bigl(\frac{1}{m_i}-\frac{q^{m_i}}{m_{i+1}}\Bigr).
\end{align}
We define the sequence of {\it multipliers} associated to the strategy $(k,\um)$ as the vector $\upi=(\pi_1,\dots,\pi_k)\in \N^k$ given by the ratios of the pool sizes,  
\begin{align}
\label{multipliers}
\pi_j:=\frac{m_{k-j+1}}{m_{k-j+2}}, \qquad 1\le j\le  k,
\end{align}
and define $\pi_0=1$.
We thus have $\um=(m_1,\dots,m_k)=(\pi_k\cdots\pi_1,\dots,\pi_2 \pi_1,\pi_1)$. In the sequel we will use both $\um$ or $\upi$, according to convenience.

\begin{Lemma}
  \label{p1}
Let $k\ge 1$ and assume $(k,\um)$ is optimal for $p$. 
Then, for all $i \in\{ 1, \ldots, k\}$, we have
$
  \frac{1}{m_i} - \frac{q^{m_i}}{m_{i+1}} \leq 0.
$
\end{Lemma}
\begin{proof}
  Suppose that the thesis  is not true. Therefore, there exists $j \in \{1, \ldots, k\}$ such that 
\begin{align} \label{nottrue}
   \frac{1}{m_j} - \frac{q^{m_j}}{m_{j+1}} > 0.
\end{align}
Obviously, $j$ must be smaller than $k$, otherwise the cost of $(m_1, \ldots, m_{k-1})$ would be smaller than the 
  cost of $(m_1, \ldots, m_k)$. 
Therefore, 
\begin{align}
D_k(\um,p) &=   1 +  \sum_{i=1}^{j-1} \left(\frac{1}{m_i} - \frac{q^{m_i}}{m_{i+1}} \right) + 
    \left(\frac{1}{m_j} - \frac{q^{m_j}}{m_{j+1}} \right) +   \sum_{i=j+1}^k \left(\frac{1}{m_i} - \frac{q^{m_i}}{m_{i+1}} \right)\notag\\
 &>   1 +  \sum_{i=1}^{j-1} \left(\frac{1}{m_i} - \frac{q^{m_i}}{m_{i+1}} \right) + 
   \sum_{i=j+1}^k \left(\frac{1}{m_i} - \frac{q^{m_i}}{m_{i+1}} \right)\qquad\text{(by \ref{nottrue})}\notag\\
   &>1+ \sum_{i=1}^{j-2} \left(\frac{1}{m_i} - \frac{q^{m_i}}{m_{i+1}} \right) 
   +\frac{1}{m_{j-1}} -\frac{q^{m_{j-1}}}{m_{j+1}}
   + \sum_{i=j+1}^k \left(\frac{1}{m_i} - \frac{q^{m_i}}{m_{i+1}} \right)
\end{align}
because $m_j > m_{j+1}$. The right hand side of the above inequality is the cost of the sequence $(m_1,\dots,m_{j-1},m_{j+1},\dots,m_k)$.
This proves that $(m_1, \ldots, m_k)$ cannot be the optimal strategy.
\end{proof}
\begin{Lemma}[Maximal infection probability for pooling]
   \label{p2}
Let $k\ge 1$ and assume $(k,\um)$ is optimal for~$p$.
Then, 
  \begin{align} \label{cotaq}
q \geq 3^{-1/3} \approx 0.693.
\end{align}
\end{Lemma}
\begin{proof} By Lemma \ref{p1}, with
 $m_{k+1}=1$, $\frac{1}{m_k} - q^{m_k} \leq 0$. This
 means that $\min_{n\in \mathbb{N}}\left(\frac{1}{n} - {q^{n}}\right)\le 0$. Write $r = 1/q$. Then,
  $1/n - q^n < 0$   iff   $1/n < q^n$   iff  $ n > (1/q^n)$  iff  $n > (1/q)^n$   iff  $n > r^n $ iff  $r^n - n < 0$.  But $(3^{1/3})^3 - 3 = 0$ and $(r^3 - 3)' > 0$ for all 
  $r>0$, then $r^3 - 3 < 0$ for all $r < 3^{1/3}$ and $r^3 - 3 > 0$ for all $r > 3^{1/3}$. 
    Now, $r^n - n > 3^{n/3} - n$ for all $r > 3^{1/3}$, therefore, it is enough to prove that
   $3^{n/3}-n > 0$ for all $n \neq 3$. This fact can be verified by direct calculation for 
  $n \leq 4$. For $n > 4$ it can be verified, taking derivatives, that the function 
  $3^{n/3}-n$ is increasing. Therefore the required inequality is verified for all $n \neq 3$ as
  we wanted to prove.
\end{proof}
\begin{Lemma}[Bounds for $m_1$ and $m_2$]
\label{p3}
Let $k\ge 1$ and $(k,\um)$ be a nested strategy.
Then,
\begin{align}
  \label{10.10}
  D_k(m,p) \leq (<)\, D_{k-1}((m_2 ,\ldots, m_k),p) &\quad\text{ if and only if }\quad m_2 |\log(q)| \leq (<) \, \frac{ \log(\pi_{k})}{\pi_{k}}, 
 \end{align}
and any of the statements in \eqref{10.10} implies
 \begin{align}
 \label{10.1}
 m_2 |\log(q)| \leq (<)\, \log(3)/3 \approx 0.366. 
\end{align}
In particular, if $D_{k+1}((m_0, m_1,\ldots, m_k),p) \geq\, D_{k}(m,p)$ for every multiple $m_0$ of $m_1$, then
\begin{align}
  \label{33p}
  \frac{\log(3)}{3|\log(q)|}\le m_1 .
\end{align}
\end{Lemma}
\begin{proof}
Recalling $\pi_k= m_1/m_2$ we have $\pi_k\ge 2$ and by (\ref{ncosto4}),
\begin{align}
  \label{34pi}
  D_{k}(m_1,\ldots, m_k,p)- D_{k-1}(m_2, \ldots, m_k,p) &= \frac{1}{m_1} -   \frac{q^{m_1}}{m_2}= \frac{1}{\pi_{k} m_2} -   \frac{q^{\pi_{k} m_2}}{m_2}.                                                  
\end{align}
To get the equivalences in \eqref{10.10} note that
\[
\frac{1}{\pi_k m_2} - q^{\pi_k m_2}{m_2} < 0\, (>0)
\]
if and only if
\begin{align}
  \label{33pi}
  m_2 |\log(q)|- \frac{ \log(\pi_{k})}{\pi_{k}} < 0\, (>0).
\end{align}
 This result is obtained by algebraic manipulation using that $\log(q) < 0$. 
Since the maximum of $\log(z)/z$ for $z \in \{2, 3, \ldots\}$ occurs   at $z = 3$, the second inequality in \eqref{10.10} implies \eqref{10.1}. On the other hand,  if $(k,\um)$ cannot be expanded (i.e., another stage cannot be added) to reduce the cost \eqref{ncosto4}, we have that \eqref{10.1} fails for $(k+1,(m_0,\dots,m_k))$, implying $m_1 \alogq \geq  \log(3)/3$, which is \eqref{33p}.
\end{proof}

\begin{Lemma}[Last multiplier is $3$ or $4$]
  \label{p4}
Let $k\ge1$ and assume that   $(k,\um)$ is optimal for $p$. Then
\begin{align}\label{3o4}
\pi_k = 3 \mbox{ or } \pi_k = 4.
\end{align}
Moreover, 
\begin{align}\label{lowbou}
q^{m_1} \ge 3^{-4/3} \approx 0.231 \iff m_1\le \frac{\frac{4}{3}\log(3)}{|\log(q)|}.
\end{align}
\end{Lemma}
\begin{proof}
The cost associated to $(k,(m_1, m_2,\ldots, m_k))$ is the cost associated to $(k-1,(m_2, \ldots,m_k))$ plus $V(m_1)$, where
\begin{align*}
V(m_1) =  \frac{1}{m_1} -   \frac{q^{m_1}}{m_2} = \frac{1}{ \pi_k m_2} -   \frac{q^{\pi_k m_2}}{m_2}   
  = \frac{1}{m_2}  \left( \frac{1}{ \pi_k } -   q^{\pi_k m_2} \right).  
\end{align*}
As $\um$ is a minimizing sequence of \eqref{ncosto4}, $V(m_1)\le0$ and $\pi_k$ is given by
 \begin{align} \label{minive}
\pi_k= \argmin\bigl\{ \sfrac{1}{ x } -   q^{x m_2}:  {x=2,3,\dots}\text{ and }\sfrac{1}{x}-q^{x m_2}\le 0\bigr\}.
\end{align}
   Let us write
$
q^{x m_2} = e^{m_2(\log q ) x} = e^{- m_2 \alogq x},\, a = m_2 \alogq.
$
With this notation, finding $\pi_k$ that satisfies \eqref{minive} reduces to 
\begin{align} \label{probax}
\mbox{ Minimize  }  h_a(x)=\frac{1}{ x } -  e^{-a x}  \mbox{  for } x = 2, 3, \dots\,\ \mbox{ and }\  \frac{1}{x}-e^{-ax}\le0.
\end{align}       
 By  elementary calculus, this optimization problem, when solved over the real numbers, reduces to 
finding $x\ge 0$ such that $e^{ax}=ax^2$ and $e^{ax}\ge x$. The solution to \eqref{probax} is found by comparing the values of $\frac{1}{x}-e^{-ax}$ at the integers that are closest to this real solution. 

Let $a_1\approx 0.067836$ and $a_2\approx 0.1323239$ be the values of $a \in (0,1/e)$ such that $h_a(5)=h_a(4)$ and $h_a(4)=h_a(3)$, respectively.
We find that
\begin{enumerate}
\item the problem \eqref{probax} has no solution if $a \ge 1/e$. 
  This corresponds to the case
  when, in fact, the sequence $(m_2, \ldots, m_k)$ cannot be further expanded.

\item $x = 2$ is never a solution of \eqref{probax}.

\item If $a \in (a_2, 1/e)$ the unique solution of (\ref{probax}) is $x = 3$.

\item If $a \in [a_1, a_2)$ the unique solution of (\ref{probax}) is $x = 4$.

\item If $ a=a_2$ both $x=3$ and $x=4$ are solutions  of (\ref{probax}), with the same cost. 

\item If $a< a_1$ the product $ a x\le 0.3031=a_1\tilde{x}$, where $\tilde{x}$ is the solution of \eqref{probax} for $a=a_1$. The inequality follows from observing that the solution to $e^z=z^2/a$ is increasing in $a$.
\end{enumerate}        
Now $m_1|\log(q)|=ax$ and Lemma \ref{p3} imply $ax\ge \log(3)/3> 0.3031$, ruling out item (6) above. This finishes the proof.
\end{proof}

 \begin{Lemma}[Last two multipliers cannot be 4 and 4]
  \label{p5}
  Let $k\ge 2$, assume that $(k,\um)$ is optimal for~$p$ and that $\pi_k = 4$. Then
$
\pi_{k-1} \neq 4.
$
\end{Lemma}
\begin{proof}
  Since $\pi_k = 4$, the analysis of the function   $\frac{1}{ x } -  e^{-a x}$ provided in the proof
  of Lemma \ref{p4} implies that
\begin{align} \label{a14}
a = m_2 \alogq \quad\text{ satisfies }\quad
 a \in [a_2,a_1]\approx [0.0678, 0.1323].
\end{align}  
Since $(k,\um)$ is optimal, \eqref{10.1} in Lemma \ref{p3} implies that  
  $m_1 \alogq > 0.366$. Thus, as $m_1 = 4 m_2$, it turns out that 
  $m_2 \alogq > 0.366/4 = \log(3)/12 = 0.0915$. Therefore, 
by \eqref{a14}, 
\begin{align} \label{nkl1}
m_2 \alogq \in [0.0915, a_1].
\end{align}

Recall that we follow the convention that $m_{k+1}=1$, so that  $m_3$ is well defined for $k\ge 2$.    
 By the argument leading to \eqref{nkl1}, if $\pi_{k-1} = 4$, 
$
4 m_3 \alogq \in [0.0915, a_1].
$
After some elementary manipulation this implies that 
\begin{align} \label{9656}
    q^{ m_3}  \in [0.9656, 0.9775].
\end{align}   
 We will show that, in all the situations in which (\ref{9656}) holds, 
\begin{align}\label{issmaller} 
D_{k+1}(27 m_3, 9m_3, 3m_3, m_3,\dots,m_k) < D_k(16 m_3, 4m_3, m_3, \dots, m_k).
\end{align}
By (\ref{ncosto4}), (\ref{issmaller}) holds if and only if 
\begin{align*}
  &\left(\frac{1}{3 m_3} + \frac{1}{9 m_3} +   \frac{1}{27 m_3}       \right) - 
            \left( \frac{q^{3 m_3}}{m_3}  + \frac{q^{9 m_3}}{3 m_3} 
   +    \frac{q^{27 m_3}}{9 m_3}             \right) \nonumber\\ 
&\hspace{3.5cm}<    \left(  \frac{1}{4 m_3} + \frac{1}{16 m_3}\right) - 
            \left(  \frac{q^{4 m_3}}{m_3}  + \frac{q^{16 m_3}}{4 m_3}           \right).  
\end{align*}
This is equivalent to
\begin{align*}
  \left( \frac{1}{3 } + \frac{1}{9 } + \frac{1}{27} \right) - 
             \left( q^{3 m_3}  + \frac{q^{9 m_3}}{3 } 
   +    \frac{q^{27 m_3}}{9 }             \right)     
<    \left(  \frac{1}{4 } + \frac{1}{16 }\right) - 
            \left(  {q^{4 m_3}}  + \frac{q^{16 m_3}}{4 }           \right).  
\end{align*}
Defining $x = q^{m_3}$ this is equivalent to
\[
  \left( \frac{1}{3 } + \frac{1}{9} + \frac{1}{27 }\right) - 
             \left( x^3  + \frac{x^9}{3} + \frac{x^{27}}{9 } \right)      
<    \left(  \frac{1}{4 } + \frac{1}{16 }\right) - 
             \left(  x^4  + \frac{x^{16} }{4 }           \right).   
\]

This inequality holds for $x \in [0, 1]$ 
 if and only if $x \in (0.9407831, 1)$. In particular, it holds for all $x \in   [0.9656, 0.9775]$. By \eqref{9656} and \eqref{issmaller}, this implies that the $\um=(m_1,\dots,m_k)$ is not optimal, a contradiction that follows from the assumption that $\pi_k=\pi_{k-1}=4$.
\end{proof}

  In the following Lemma we prove that the sequence of multipliers associated to an optimal strategy is non-decreasing. 
  \begin{Lemma}[Multipliers are non decreasing]
    \label{p7}
    Let $k\ge 2$ and assume $(k,\um)$ is optimal for $p$.
    Then
\begin{align*}
\pi_1 \le \dots \le \pi_k.
\end{align*} 
  \end{Lemma}
  \begin{proof}
  Let us first show that under the conditions of the theorem $\pi_{k-1}\le\pi_k$. 
  
  Define $\utpi=\pi_{k-1}\wedge \pi_k$ and $\otpi=\pi_{k-1}\vee\pi_k$, where $a\wedge b:=\min\{a,b\}$ and $a\vee b:=\max\{a,b\}$. Notice that by Lemma \ref{p4}, $\utpi\le 4$, hence $\utpi \in \{2,3,4\}$. Suppose that 
  $\utpi<\otpi$, as otherwise $\utpi=\otpi$ and $\pi_{k-1}=\pi_k$.
  
 We wish to compare $D_k^{\text{nat}}$, the cost associated with the sequence of multipliers
$(\pi_1, \ldots, \utpi, \otpi)$ (last two multipliers in increasing order) with $D_k^{\text{inv}}$, the cost associated with the
 sequence
  $(\pi_1, \ldots, \otpi, \utpi)$ (last two multipliers in decreasing order).

From \eqref{ncosto4}, we have 
\begin{align}
\label{nat-inv}
D_k^{\text{inv}}-D_k^{\text{nat}}
=  \frac{1}{\otpi m_3}   -  \frac{q^{\otpi m_3}}{m_3} 
- \frac{q^{m_1}}{m_3 \otpi} 
- \frac{1}{\utpi m_3 }   +  \frac{q^{\utpi m_3}}{m_3}    
+\frac{q^{m_1}}{m_3 \utpi}.
\end{align}
This expression is positive if and only if 
$
(\otpi - \utpi)(q^{m_1} - 1) + \utpi \,\otpi( q^{m_1/\otpi} - q^{m_1/\utpi})
$
is positive. 

Let us call $z = q^{m_1}$. We are interested in
  the sign of the expression
 $(\otpi- \utpi)(z - 1) + \utpi \,\otpi( z^{1/\otpi}\, - \,z^{1/\utpi})$    
 when $1 \geq z > 3^{-4/3} \approx 0.231$ and $\utpi=2, 3 \text{ or } 4$.

 If $\utpi = 2$ then, by Lemma \ref{p4}, $\pi_{k-1}=2$ and $\pi_k=\otpi$. 
 We will nonetheless compare $D_k^{\text{nat}}$ and $D_k^{\text{inv}}$ as the argument will be necessary for the study
 of the multipliers $\pi_j,\,j<k$. In this case we need to minimize
  $
 (\otpi - 2)(z - 1) + 2 \otpi ( z^{1/\otpi} - z^{1/2})
$   
 subject  to $z \in [0.231, 1]$ and $\otpi\ge 3$.         

If $\utpi = 3$, we have that 
  $\otpi \geq 4$. Therefore, the problem in two variables is:
$
\mbox{Minimize }  (\otpi - 3)(z - 1) + 3 \otpi( z^{1/\otpi} - z^{1/3})
$
subject to $z \in [0.231, 1]$ and $\otpi \ge 4$.

Finally, if $\utpi= 4$, the relevant optimization problem in two variables is:  
 $
\mbox{Minimize }  (\otpi - 4)(z - 1) + 4 \otpi ( z^{1/\otpi} - z^{1/4})
$
subject  to $z \in [0.231, 1]$ and $\otpi \ge 5$.
 
For $z\in [0.231,1]$ and values of $\otpi$ in the range considered in each of the problems, the minimum is achieved at $z = 1$, 
and at this value the objective function vanishes. This proves that the expression in \eqref{nat-inv} is non-negative, and hence $\pi_{k-1}\le \pi_k$.

The argument proceeds by induction. Suppose now that $\pi_j<\pi_{j+1}<\dots<\pi_k$. Then $\pi_j\in \{2,3,4\}$ and $q^{m_{k-j+1}}\ge q^{m_1}\ge 3^{-4/3}$. 
Let $\utpi_j=\pi_{j-1}\wedge \pi_j$ and $\otpi_j=\pi_{j-1}\vee\pi_j$. We may assume that $\utpi_j\neq \otpi_j$, otherwise there is nothing to prove. We wish to compare the following expressions
\begin{align*}
D_k^{\text{nat}}(j)&:=D_k\big({\small (m_1,\dots,m_{k-j}, \otpi_j\,\utpi_j m_{k-j+3},\utpi_jm_{k-j+3}, m_{k-j+3},\dots,m_k)},p\big),\\
D_k^{\text{inv}}(j)&:=D_k\big({\small (m_1,\dots,m_{k-j}, \utpi_j\,\otpi_j m_{k-j+3},\otpi_jm_{k-j+3}, m_{k-j+3},\dots,m_k)},p\big).
\end{align*}

A simple computation shows that $D_k^{\text{inv}}(j)-D_k^{\text{nat}}(j)$ has the same sign as 
\begin{align}
\label{upiopi}
\big(\otpi_j - \utpi_j\big)(z - 1) + \utpi_j\,\otpi_j\big( z^{1/\otpi_j} - z^{1/\utpi_j}\big),
\end{align}
with $z:=q^{m_{k-j+1}} \in [3^{-4/3},1]$ and $\utpi_j\in \{2,3,4\}$. The previous computations show that in the given range for $z$ and $\utpi_j$ \eqref{upiopi} is nonnegative,  $D_k^{\text{inv}}(j)-D_k^{\text{nat}}(j)\ge 0$, and therefore $\pi_{j-1}=\utpi_j$, $\pi_j=\otpi_j$.
  \end{proof}
   
\begin{Lemma}[Successive multipliers cannot be (2, 2)]
  \label{p9}
  Let  $k\ge 2$ and assume $(k,\um)$ is optimal for~$p$. Then, for all $j = 1, \dots, k-1$, $(\pi_j, \pi_{j+1}) \neq (2, 2)$.
\end{Lemma}
\begin{proof}
Suppose there is $1\le j\le k-1$ such that $(\pi_j,\pi_{j+1})=(2,2)$. Then 
$\um=(m_1,\dots,m_k)=(m_1,\dots, m_{k-j-1}, 4m_{k-j+2},2m_{k-j+2}, m_{k-j+2},\dots,m_k)$. Consider the nested sequence obtained by removing the $(k-j+1)$-th stage,
$\um^j=(m_1,\dots, m_{k-j-1}, 4m_{k-j+2}, m_{k-j+2},\dots,m_k)$. Then
\begin{align*}
D_k(\um,p)-D_{k-1}(\um^j,p)&=\frac{1}{2m_{k-j+2}}-\frac{q^{2m_{k-j+2}}}{m_{k-j+2}}-\frac{q^{4m_{k-j+2}}}{2m_{k-j+2}}+\frac{q^{4m_{k-j+2}}}{m_{k-j+2}}\\[2mm]
&=\frac{1}{2m_{k-j+2}}-\frac{q^{2m_{k-j+2}}}{m_{k-j+2}}+\frac{q^{4m_{k-j+2}}}{2m_{k-j+2}}.
\end{align*}
This expression has the same sign as 
\begin{align*}
\frac{1}{2} -q^{2m_{k-j+2}}+\frac{q^{4m_{k-j+2}}}{2}.
\end{align*}
Writing $x=q^{2m_{k-j+2}}$, we see that we only need to study the sign of $\frac12-x+x^2$, $x\in[0,1]$. Since this function is positive over this interval, we conclude that replacing 
the strategy $(k, \um)$ by $(k-1,\um^j)$ reduces the cost, a contradiction to the optimality of $(k,\um)$.
\end{proof}

\begin{Lemma}[Only the first multiplier could be 2]
  \label{p10} 
  Let $k\ge 2$ and $(k,\um)$ be optimal for $p$. Then $\pi_j \neq  2$ for all  $j > 1$.
\end{Lemma}
\begin{proof}
    Suppose that $j > 1$ and $\pi_j = 2$. By Lemma \ref{p7}, $\pi_1 = \ldots= \pi_{j-1} = 2$. But by Lemma \ref{p9} this is impossible.
   Therefore, the existence of $j > 1$ with $\pi_j = 2$ leads to a contradiction.
  \end{proof}
  \begin{Lemma}[All but the first and last multipliers must be 3]
    \label{p11}
    Let $k \geq 3$ and $(k,\um)$ be optimal for~$p$. Then, $\pi_j = 3$ for all  $j = 2, \ldots, k-1$.
  \end{Lemma}
  \begin{proof}
    By Lemma~\ref{p4}, $\pi_k=3$ or $4$.  Then by Lemma~\ref{p7}, $\pi_{k-1}\le 4$, and Lemma~\ref{p5} implies $\pi_{k-1} \le 3$. 
    By Lemma~\ref{p10}, $\pi_{k-1}\neq 2$. Then,
  $\pi_{k-1} = 3$, and by Lemma~\ref{p7}, $\pi_{j} \leq 3$ for all $j = 1, \ldots, k-2$. Therefore, by Lemma~\ref{p10}, 
  $\pi_j = 3$ for $j = 2, \ldots, k-2$. This completes the proof.
  \end{proof}
  The following result summarizes the information on the optimal strategy  collected so far, and together with Lemma \ref{p2}, they complete the proof of Theorem \ref{1of4}.
  \begin{Proposition}[The four possible optimal strategies]
    \label{p12}
     Let $(k,\um)$ be optimal for $p$.   If $k = 1$, then $\um = \upi = (3)$ or $\um= \upi = (4)$. Otherwise, 
\begin{align} \label{semifinal}
\upi = (2 \mbox{ or } 3, \; 3, \dots, 3, \; 3  \mbox{ or } 4).
\end{align}
  \end{Proposition}
  \begin{proof}
  Follows from Lemmas \ref{p4}--\ref{p11}.
  \end{proof}

 Building upon the previous proposition, the following result establishes the number $k$ of pooled stages associated to an optimal strategy, which is expressed in terms of the probability $p$ and the particular form of the strategy. 
  \begin{Proposition}[Optimal length $k$]
    \label{p13} Let $p\le \rho_0$ and $(k,\um)$ be a strategy such that
    \begin{align}
      \label{ni+ni-}
      D_k(\um,p)\le D_{k-1}((m_2,\dots,m_k),p)\wedge D_{k+1}((m_0,\dots,m_k),p),
    \end{align}
for any multiple $m_0$ of $m_1$. Then,\vspace{2mm}

\noindent 1. if $\upi = (2, 3, \dots, 3)$, then
 \begin{align} \label{floor1}
k=k_{23} := \bigg\lfloor   \frac{\log(\log(3)) - \log(2)}{\log(3)} - \frac{ \log(\alogq)}{\log(3)}  + 1 \bigg\rfloor
=\bigg\lfloor   \log_3\big(\frac{1}{|\log_3(q)|}\big)-\log_3(2)+1\bigg\rfloor,
\end{align}  

\noindent 2. if $\upi = (2, 3, \dots, 3, 4)$, then 
 \begin{align} \label{smallsize1}
 k=k_{24} \in \Bigl(\log_3\bigl(\frac{1}{|\log_3(q)|}\bigl)-\log_3(8)+1,\;\; \log_3\bigl(\frac{1}{|\log_3(q)|}\bigl)-\log_3(8)+2+\log_3(\log_3(4))\Bigr] \cap \N,
 \end{align}

\noindent 3. if $\upi = (3, 3, \dots, 3)$, then
 \begin{align} \label{floor2}
k=k_3:= \bigg\lfloor   \log_3\big(\frac{1}{|\log_3(q)|}\big)\bigg\rfloor,
\end{align}  

\noindent 4. if $\upi = (3, 3, \dots, 3, 4)$, then
    \begin{align}
      \label{smallsize2}
k=k_{34}\in \left(\log_3\big(\frac{1}{|\log_3(q)|}\big)-\log_3(4),\;\; \log_3\big(\frac{1}{|\log_3(q)|}\big)-\log_3(4)+1+\log_3(\log_3(4))\right]\cap \N
    \end{align}
 \end{Proposition}
 \begin{proof}
   Let us prove (\ref{floor1}). 
   Lemma~\ref{p3} and the fact that $m_1=3 m_2$ implies \eqref{l3m} below, which in turn is equivalent to each line following it up to \eqref{thek}
   \begin{equation}
     \label{l3m}
     \log(3)     \; > \;  m_1 |\log(q)| \; \geq \; \log(3)/3
   \end{equation}
 \[
  \log(3)      \; > \;  2 \times 3^{k-1} \alogq  \;\geq \; \log(3)/3
\]  
 \[
  \log(\log(3))      \;> \; \log(2) + (k-1) \log(3) + \log(\alogq)  \;\geq \; \log(\log(3)/3)
\]   
 \[
   \log(\log(3)) - \log(2) - \log(\alogq)
   \;> \;
   (k-1) \log(3)
   \;\geq \;
   \log(\log(3)/3) - \log(2) - \log(\alogq)
\]   
\begin{equation} \label{thek}
  \frac{ \log(\log(3)) - \log(2)}{\log(3)} -  \frac{ \log(\alogq)}{\log(3)} + 1
  \;  > \;
  k
  \;  \geq \; \frac{\log(\log(3)) - \log(2)}{\log(3)} - \frac{ \log(\alogq)}{\log(3)},
\end{equation}   
which is \eqref{floor1}. We may prove (\ref{smallsize1}), 
 (\ref{floor2}), and \eqref{smallsize2} using a similar sequence of arguments. 
\end{proof}

We will later use the following a priori bounds for $p$ in terms of the optimal strategy.
\begin{Corollary}[A priori bounds for $p$]
  \label{c13}
  Let $(k,\um)$ be a nested strategy satisfying \eqref{ni+ni-}. Then,
\begin{align} 
  &\text{if $\upi = (3, 3, \dots, 3,3)$, then } \quad 1- (3^{- 1/9})^{1/3^{k-1}} \le p \le1- (3^{-1/3})^{1/3^{k-1}};
    \label{floor2q}\\
  & \text{if $\upi = (3, 3, \dots, 3, 4)$, then } \quad  1- (4^{-1/12})^{1/3^{k-1}}\le  p\le 1- (3^{-1/4})^{1/3^{k-1}}. \label{smallsize4q}
\end{align}  
\end{Corollary}
\begin{proof}
The statements follow from \eqref{floor2} and \eqref{smallsize2}, respectively.
\end{proof}

\subsection{Conjectured optimal strategies}
\label{conjectured}
We conjecture that there are only two families of optimal strategies. 
\begin{Conjecture}[Conjectured optimal strategy]
   \label{c8}
   If $p\ge 1-3^{-1/3}$, then the optimal strategy is to test all individuals (no pooling). If $p\le 1-3^{-1/3}$,  then there is a $k=k(p)\ge1$ and a strategy~$(k,\um)$ optimal for~$p$ satisfying
   \begin{align}
\label{c88}
     (k,\um) &\in\bigl\{(k,\um_{33}),(k,\um_{34})\bigr\}.
\end{align}
\end{Conjecture}
In Conjecture \ref{p14} below we rephrase this statement and corroborate that it holds for all $p>2^{-51}$.  

The transition between the strategies in \eqref{c88} occurs at points $\lambda_k$ and $\rho_k$, where $\rho_0:= 1-3^{-1/3}$ is the solution of $D_1(m_{33},p) = 1$, and for  $k\ge 1$, 
 \begin{align}
     \lambda_k&:= \text{ solution $p$ in $[0,\rho_{k-1})$ of } D_k(\um_{33},p)= D_k(\um_{34},p);\label{lk}\\[2mm]
    \rho_{k}&:= \text{ solution $p$ in $[0,\lambda_k)$ of } D_{k+1}(\um_{33},p)= D_{k}(\um_{34},p).\label{rk}
 \end{align}
 The solution of each of these equations is unique in the corresponding interval. In Lemma \ref{laro1} we show that for each $p\le \rho_0$:
 \begin{align}
    \label{c889q}
   \min_{j,\ell\ge1}(D_j(\um_{33},p)\wedge D_\ell(\um_{34},p)) \text{ is realized by }
    \begin{cases}
      (k,\um_{33})& \text{if } \lambda_{k}\le p \le \rho_{k-1},\\
       (k,\um_{34})& \text{if } \rho_{k}\le p \le \lambda_k.
  \end{cases}
 \end{align}
 where we recall that $a\wedge b$ denotes $\min\{a,b\}$.
The first transition points and the cost of the conjectured optimal strategy as a function of $p$ in log-log scale are shown in Fig.~\ref{costo-conjeturado}. 

{ \noindent
\begin{minipage}{.4\linewidth}
  \begin{center}
\begin{tabular}{|c|l|l|}
\hline 
 $k$ &   $\lambda_k\approx$& $\rho_{k-1}\approx$   \\
\hline
 1 &   0.1239 & 0.3066   \\
 \hline
 2 &  0.0431 & 0.1098  \\
\hline
3 &  0.0145 & 0.0380 \\
\hline
4 &    0.0048 & 0.0128   \\
\hline
5 & 0.0016 &  0.0043  \\
\hline
6 &  0.0005 &0.0014  \\
\hline
\end{tabular}
\end{center}
\end{minipage}%
\begin{minipage}{.6\linewidth}
	 \includegraphics[height=7cm, trim={0mm 0mm 0 0mm},clip] 
              {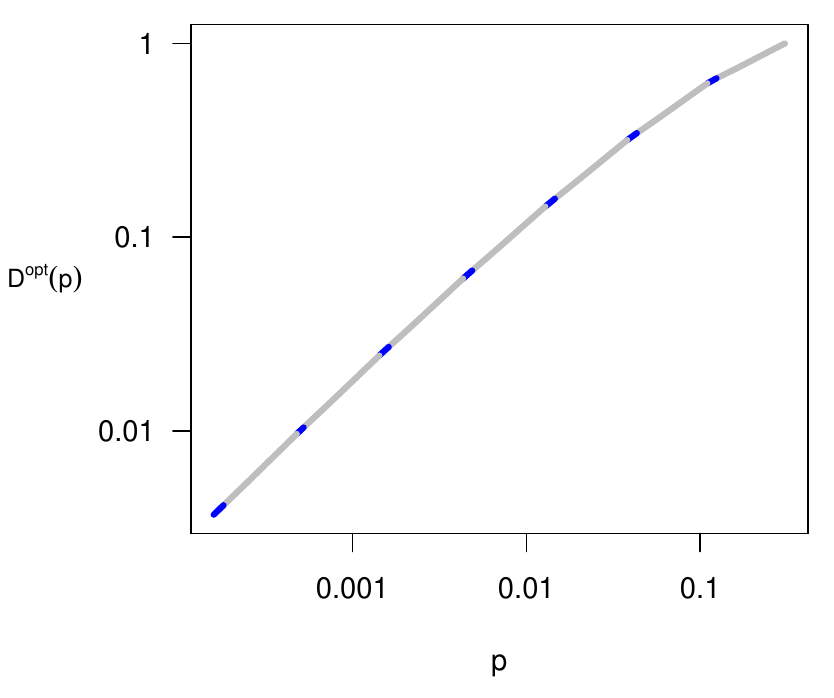}
\end{minipage}%
 \captionof{figure}{Left: First transition points $\lambda_k$ and $\rho_k$ defined in \eqref{lk} and \eqref{rk}. Right: Cost $\Dopt(p)$ of the conjectured optimal strategy as function of $p$ in log-log scale. Short (blue) segments correspond to $\um_{34}$ and long (grey) segments to $\um_{33}$.
                \label{costo-conjeturado}\\}

}

We have checked that for a large range of values of $p$ the optimal strategy has $\pi_1=3$:
 \begin{Conjecture}[Corroborated for $p>2^{-51}$]
    \label{p14}
Let $p\in(2^{-51},1-3^{-1/3})$ and $(k,\um)$ be optimal for~$p$. Then, if $k = 1$, $\um=\upi = (3)$ or $\um=\upi =(4)$. Otherwise, 
$
\upi = (3, 3, \ldots, 3, 3 \mbox{ or } 4). 
$
\end{Conjecture}
\begin{proof}[Corroboration]
  Given $q \in [3^{-1/3}, 1)$ we proceed as follows:
\begin{enumerate}
\item Find $k = k_{23}(p)$ using (\ref{floor1}) and define $D_{23}(p)$ as the cost of using the strategy defined by $k = k_{23}(p)$, $\pi_1 = 2$ and $\pi_k= 3$. 
\item If the right-hand side of (\ref{smallsize1}) is nonempty, compute $k_{24}(p)$ as the unique integer that satisfies (\ref{smallsize1}) and define  
  $D_{24}(p)$ as the cost of using the strategy defined by $k = k_{24}(p)$, $\pi_1 = 2$ and $\pi_k = 4$. Otherwise, define $D_{24}(p) = \infty$. 
\item Proceed as in the previous steps and use \eqref{floor2} and \eqref{smallsize2} to determine, respectively, $k_3(p)$ and $k_{34}(p)$, if the latter is well defined. Define $D_{33}(p)$ and $D_{34}(p)$, respectively, as the costs of using the strategy defined by $k=k_3(p)$, $\pi_1=3$ and $\pi_k =3$, and the one defined by $k=k_{34}(p)$, $\pi_1=3$ and $\pi_k =4$, when $k_{34}(p)$ exists.
\item Define $\Phi(p) = \min\{D_{33}(p), D_{34}(p)\} - \min\{D_{23}(p), D_{24}(p)\}$.
\end{enumerate}
Given $p$ in the allowed domain, if  $\Phi(p) <0$ the strategies with $\pi_1 = 3$ are better than the strategies with $\pi_1=2$. 
  The property $\Phi(p) < 0$ has been corroborated for $p \geq 2^{-51}$ using multiple precision computations and 
  careful floating point analysis for the smallest values of $p$. Walter Mascarenhas showed that the computations for p = $2^{-51}$ may be processed without rounding errors confirming the conjecture that 
$\min\{D_{33}, D_{34}\} < \min\{D_{23}, D_{24}\}$ for this value of $p$. The result follows from Proposition \ref{p12}.
\end{proof}

\begin{Remark}\rm
  By Proposition~1 the conjecture reduces to saying that, at an optimal strategy, we have that
  $\pi_1 \neq 2$, which is equivalent to saying that  
  $\pi_1 = 3$ or $\pi_1 = 4$. As a consequence, the conjecture is that $\Phi(p) < 0$ for all $p$ smaller than
  $1-3^{-1/3}$. The ``corroboration" for $p \geq 2^{-51}$ is as follows: For $p = 2^{-51}$ floating point 
  calculations that lead to the computation of $\Phi(p)$ can be performed without errors thanks to a careful 
   organization of floating point computations provided by Prof. W. Mascarenhas. In this way, $\Phi(2^{-51}) < 0$
  is exactly verified. On the other hand, $\Phi(p)$ decreases with respect to $p$, so that $\Phi(p) < 0$ follows
  for $p > 2^{-51}$. Moreover, computations of $\Phi(p)$ and consequent verification of $\Phi(p)<0$ can be 
  verified using resources of Interval Computations. 
\end{Remark}

\begin{Lemma}[A priori transition points]
  \label{dan12}
Let $k\ge 1$, $\lambda_k,\rho_k$ as defined in \eqref{lk}-\eqref{rk} and $p<\rho_{k-1}$. Then,
    \begin{align}
\text{\rm sign}\bigl(D_k(\um_{33},p) -D_k(\um_{34},p)\bigr)& = \text{\rm sign}(\lambda_k-p), \label{dan13}\\[2mm]
\text{\rm sign}\bigl(D_k(\um_{34},p) -D_{k+1}(\um_{33},p)\bigr)& = \text{\rm sign}(\rho_k-p)
      .\label{dan14}
    \end{align}
Define the functions $F,G:(0,1)\to\R$ by
    \begin{align}
      \label{FG7}
      F(a)&:= \sfrac1{12} -e^{-3a} +  e^{-4a},  \qquad G(a):= - \sfrac7{36}
    - e^{-4a}   + \sfrac13 e^{-9a}   + e^{-3a}.
    \end{align}
    Let $\alpha_1$ be the smallest root of $F$ and $\beta$ the only root of $G$. Then
    \begin{align}
      \label{laroex}
      \lambda_k =   1-e^{-\alpha_1/3^{k-1}} ,\qquad \rho_k =  1-e^{-\beta/3^{k-1}}.
    \end{align}
\end{Lemma}

\begin{proof} Denoting   $a:= 3^{k-1} |\log (1-p)|$ we have 
\begin{align}
    D_k(\um_{33},p) -D_k(\um_{34},p)
    &= \sfrac1{3^{k-1}}F(a);\label{dan17}\\[2mm]
    D_k(\um_{34},p) -D_{k+1}(\um_{33},p)
    &=  \textstyle{\frac1{3^{k-1}}}\,G(a).   \label{dan16}
  \end{align}
  Elementary computations show that $G$ is continuous, strictly decreasing in $[0,1]$ and $\beta\approx 0.1164$ is its unique root in $(0,1)$. Hence, $G(a)$ is positive for $a<\beta$ and negative for $a>\beta$, which in turn implies that \eqref{dan14} holds for all $k\ge1$ with $\rho_k =  1-e^{-\beta/3^{k-1}}$.
  
Similarly, $F$ is continuous, has roots at $\alpha_1\approx 0.1323$ and $\alpha_2\approx 0.5343$, $F(a)$ is negative for $a$ in $(\alpha_1,\alpha_2)$ and positive for $a<\alpha_1$ and $a>\alpha_2$.  If $p<\rho_0=1-3^{-1/3}$, recalling the definition of $a$, we have that for $k=1$, $a = |\log(1-p)| < |\log 3^{-1/3}| \approx 0.3662 <\alpha_2$. This implies $F(a)<0$ for $a\in (\alpha_1,  0.3662)$, which implies \eqref{dan13} for $k=1$ and $\lambda_1 = 1-e^{-\alpha_1}$. Now we proceed by induction. Assume \eqref{dan13} holds for all $j\le k-1$ and take $k\ge 2$ and $p<\rho_{k-1}$. By \eqref{dan14}, the cost of $(k,\um_{33})$ is smaller than the cost of $(k-1,\um_{34})$, which is smaller than the cost of $(k-1,\um_{33})$, by the inductive hypothesis. This and  \eqref{10.1} imply $a=3^{k-1}|\log(1-p)| \le(\log 3)/3\approx 0.3662<\alpha_2$. Hence, if $\alpha_1<a<(\log 3)/3$ then $F(a)<0$, while if $a<\alpha_1$ then $F(a)>0$, and solving for $p$ we get that \eqref{dan13} holds for $k$ with $\lambda_k =  1-e^{-\alpha_1/3^{k-1}}$.
\end{proof}

\begin{Lemma}[Transition points]
  \label{laro1}
  For each $p\le \rho_0$ we have
\begin{align}
    \label{c88q}
  \min_{j,\ell\ge1}\bigl(D_j(\um_{33},p) \wedge D_\ell(\um_{34},p)\bigr) \,=\, 
   \begin{cases}
   D_k(\um_{33},p)& \text{if } \lambda_{k}\le p \le \rho_{k-1},\\
       D_k(\um_{34},p)& \text{if } \rho_{k}\le p \le \lambda_k.
  \end{cases}
\end{align}
\end{Lemma}
This lemma computes the transition points between the optimal conjectured trajectories. For instance  $k=1$ for 
  $p$ between $0.1099$ and $0.3066$, and the optimal $(k,\um)$ is $(1, (4))$ for $0.1099\le p\le 0.1239$, while the optimal choice is  $(1, (3))$ for  $0.1240\le p \le0.3066$.
\begin{proof}
  Fix $k\ge 1$ and  call $a:=3^{k-1}|\log q|$. Let $(k', \um')$ realize the minimun in \eqref{c88q}, then it satisfies \eqref{ni+ni-} and by \eqref{floor2q} and \eqref{smallsize4q} we have 
      \begin{numcases}{\text{If }(k',\um') =}
 \text{ $(k,\um_{33})$, then }\;  &$a\in\bigl( |\log 3^{-1/9}|,|\log 3^{-1/3}|\bigr)\approx (0.1221,0.3662)$,\label{fl61} \\
  \text{ $(k+1,\um_{33})$, then }\;  &$a\in\sfrac13\bigl( |\log 3^{-1/9}|,|\log 3^{-1/3}|\bigr)\approx (0.0406,0.1221)$, \label{fl62}\\
\text{ $(k,\um_{34})$, then }\;  &$a\in\bigl( |\log 4^{-1/12}|,|\log 3^{-1/4}|\bigr)\approx (0.1155,0.2746)$, \label{ss63} \\
  \text{ $(k+1,\um_{34})$, then }\;  &$a\in\sfrac13\bigl( |\log 4^{-1/12}|,|\log 3^{-1/4}|\bigr)\approx (0.0385,0.0915)$.\label{ss64}
    \end{numcases} 
The intervals \eqref{fl61} and \eqref{fl62} are disjoint and they determine intervals for $p$ that are decreasing in $k$. Similarly \eqref{ss63} and \eqref{ss64} are disjoint and they also determine  intervals for $p$ that are decreasing in $k$. On the other hand, \eqref{ss63} is contained in the union of  the closures of  \eqref{fl61} and \eqref{fl62}. So we only need to compare $(k,\um_{34})$ with  $(k,\um_{33})$ and  $(k+1,\um_{33})$. But this has been done in Lemma \ref{dan12}.
\end{proof}
\section{Linearization of the cost function}
\label{optim}

In this section we study the linearized version of the cost, which is easier to optimize and gives a good approximation to the cost for small $p$. Only in this section, we allow the pool sizes $m_j$ and multipliers $\pi_j$ to take non-negative real values. These results will be applied in the next section to estimate the asymptotic cost of the optimal strategy. 

Let us fix $p$ and a stage number $k+1$. We linearize the expected number of tests per individual $D_{k}=D_{k}(m,p)$ obtained in \eqref{meankdep}:
\begin{align}
D_{k}
&=\frac1{m_1} + 1-e^{m_k \log q}+ \sum_{j=2}^{k} \frac1{m_j} \bigl(1-e^{m_{j-1}\log q}\bigr)
=\oD_k + {\text{error}}. \label{linear+}
\end{align}
where the linear approximation $\oD_k=\oD_k(\um,p)$ is given by
\begin{align}
\oD_k:=\frac1{m_1}+m_k p+p\sum_{j=2}^k \frac{m_{j-1}}{m_j}. \label{linear++}
\end{align}
The linearized cost $\oD_k$ coincides with the cost proposed by Finucan \cite{finucan}, who assumed that for suitable $p$ and $m_1$ there is at most one infected individual per pool at all stages; we give some details after Lemma \ref{optipool}. 

In the next lemma we show that the cost is bounded above by the linearized cost, and provide an estimate for the difference. The result is proved  in Appendix \ref{aB}.
\begin{Lemma}[Domination and error bounds]
  \label{error-bound}
  Let $p\in[0,\frac12]$.
  Let $\um=(m_1,\dots,m_k) \in \R^{k}_{\ge 1}$. Then 
  \begin{enumerate}
  \item The linearized cost is an upper bound to the cost,
  \begin{align}
  \label{negative}
  D_k(\um,p)\le \oD_k(\um,p).
  \end{align}
\item If $\frac{m_{i-1}}{m_i}\le \ell$ for $2\le i\le k+1$, with  $m_{k+1}:=1$, then
\begin{align}
\label{error++}
  |D_k(\um,p)-\oD_k(\um,p)|\le  \,\ell m_1\log^2 q+\ell k p^2.
\end{align}
In particular, when  $\frac{ m_j}{m_{j+1}} = \mu>0,\,1\le j\le k$, equation \eqref{error++} becomes
 \begin{align}
 \label{error+}
    |D_k(\um,p)-\oD_k(\um,p)|\le  \mu^{k+1} \log^2 q+ \mu k p^2.
 \end{align}
\end{enumerate}
\end{Lemma} 

Define the optimal values for $\oD_k$ by
\begin{align}
  \um^\sharp(k)=(m^\sharp_1(k),\dots,m^\sharp_k(k)) &:= \underset{(m_1,\dots,m_k)\in\R^k_+}{\arg\min} \oD_k \quad\text{and}\label{m*}\\
  \oD_k^\sharp &:= \oD_k(\um^\sharp(k),p).\label{od*}
\end{align}
In the next two lemmas we compute the optimal linearized values, see also~\cite{finucan}.
\begin{Lemma}[Optimal pool sizes]
  \label{ops1}
  Let  $p\in(0,1)$ and $k\in \N$, $k\ge 2$. Then
  \begin{align}
\um^\sharp_j(k) &={p^{-\frac{k-j+1}{k+1}}}, \qquad 1\le j\le k. \label{gpools} \\[2mm]
   \oD_k^\sharp &= (k+1)\,p^{\frac{k}{k+1}}.  \label{D*}
  \end{align} 
\end{Lemma}

\begin{proof}
For $k\ge 3$ we get
\begin{align}
&\frac{\partial\oD_k}{\partial m_1}=-\frac{1}{m_1^2}+\frac{p}{m_2},\label{d1}\\
&\frac{\partial\oD_k}{\partial m_i}=-\frac{pm_{i-1}}{m_i^2}+\frac{p}{m_{i+1}},\hspace{1.5cm}{2\le i\le k-1,} \label{di}\\
&\frac{\partial \oD_k}{\partial m_k}=-\frac{pm_{k-1}}{m_k^2}+p. \label{d3}
\end{align}
In order to find critical points we look for values of $m_j,\, 1\le j\le k$ where these derivatives vanish.  We get
\begin{align}
&\frac{\partial \oD_k}{\partial m_k}=0\iff m_{k-1}=m_k^2 \hspace{2cm}\text{from \eqref{d3}} \label{uno}\\
&\frac{\partial\oD_k}{\partial m_i}=0\iff m_{i-1}=\frac{m_i^2}{m_{i+1}} \hspace{1cm}\text{for }2\le i\le k-1, \label{dos}\\
&\frac{\partial\oD_k}{\partial m_1}=0\iff m_2=p m_1^2 \label{tres}
\end{align}
Given $m_k$ we use \eqref{uno} and \eqref{dos} to solve backwards in the index $i$, and we get
\begin{align}
&m_{k-1}=m_k^2,\quad m_{k-2}=\frac{m_{k-1}^2}{m_k}=m_k^3,\quad m_{k-3}=\frac{m_{k-2}^2}{m_{k-1}}=m_k^4 \notag\\
&\ \ \text{ and, in general,}\hspace{1cm} m_{k-j}=m_k^{j+1}, \qquad 0\le j\le k-1. \label{uno+dos}
\end{align}
We replace the values of $m_1$ and $m_2$ in \eqref{tres} to obtain the equation
\begin{align}
&m_k^{k-1}=pm_k^{2k}\iff m_k={p^{-\frac{1}{k+1}}}, \label{50.5}
\end{align} 
from where we get \eqref{gpools}. We show in Appendix \ref{hessian} that the Hessian matrix of $\oD_k$ evaluated at the critical point $(m^\sharp_1(k),\dots, m^\sharp_k(k))$ is positive definite, and hence this is a minimum of $\oD_k$.

Substituting \eqref{gpools} in \eqref{linear++} yields $ \oD_k^\sharp =p^{\frac{k}{k+1}}+{kp}\,{p^{-\frac{1}{k+1}}}=(k+1)\,p^{\frac{k}{k+1}}$.
\end{proof}

 We now optimize $\oD^\sharp_k$ as a function of $k$. Denote
 \begin{align}
   \label{od*3}
   \oD^\sharp=\oD^\sharp(p):= \min_{k\in\R_+} \oD_k^\sharp;\qquad k^\sharp :=  \underset{k\in\R_+}{\arg\min}\, \oD^\sharp_k.
 \end{align}
In general  $k^\sharp \in \R\setminus \N$. Notice that when $k$ is not a positive integer it is not possible to define a vector $(m_1,\dots,m_k)$ where to evaluate $\oD_k$.

 \begin{Lemma}[Optimal number of stages]
  \label{optipool}
  For any $p\in(0,1)$ we have
  \begin{align}
    \label{optipool23}
     k^\sharp = \textstyle{\log\frac1p -1} ,  \quad
     \oD^\sharp =  e\,p \log(1/p).
   \end{align}
 Furthermore, if $p= e^{-u}$ for some integer $u\ge 2$, then $k^\sharp = u-1$ and 
   \begin{align}
     \label{optipool3}
    \oD^\sharp = \oD_{k^\sharp}(\um^\sharp,p), \quad \text{where} \quad \um^\sharp = (e^{u-1},\dots, e).
   \end{align}
 \end{Lemma}

  \begin{proof}[Proof of Lemma \ref{optipool}]
We compute the derivative
\begin{align*}
\frac{\partial \oD^\sharp_k}{\partial k}=p^{\frac{1}{k+1}}\Big[1+\frac{\log p}{1+k}\Big],
\end{align*}
which vanishes at $k= k^\sharp=\log\frac1p -1$.
This is in fact a global minimum of $\oD^\sharp_k$, for a given value of~$p$.
We now replace this value in  $\oD^\sharp_k$ to get
\begin{align*}
     \oD^\sharp &=  p^{1-\frac1{\log(1/p)}}\,\textstyle{\log\frac1p}=e\,p\log(1/p).
\end{align*}
Under $p=e^{-u}$ we have $k^\sharp = u-1$, which  replaced in \eqref{gpools} gives
\begin{align}
  m^\sharp_{k^\sharp} &= e\quad\text{and}\quad m^\sharp_j = e^{k-j+1}.\qquad\qedhere
\end{align}
\end{proof}

\paragraph{Remark} 
Finucan \cite{finucan} proposes to iterate Dorfman's strategy with non necessarily nested pools, under the assumption that
at every stage each pool has at most one infected individual. Call $U:=$ number of infected individuals in a population of size $N$, $U$ has Binomial$(N,p)$ distribution. The number of individuals to be tested in the $i$-th stage is $Um_{i-1}$, and the total number of tests is 
\begin{align}
  \label{fin1}
  \frac{N}{m_1} + \frac{Um_1}{m_2} + \frac{Um_2}{m_3}+\dots+ \frac{Um_{k-1}}{m_k} + Um_k.
\end{align}
Dividing by $N$ and taking expectation, Finucan gets the linearized cost $\oD_k(\um,p)$ defined in \eqref{linear++} and derives the results of Lemmas~\ref{ops1} and \ref{optipool}. He also shows that these optimal values maximize the information gain per test in the case that there is at most one infected individual per pool. 

However, the hypothesis that there is at most one infected individual per pool is not satisfied for the optimal values \eqref{gpools}. Indeed, when $m_1\approx 1/p$, the number of infected individuals per pool is approximately Poisson with expectation~1. In any case, Finucan's cost provides an upper bound to the true  cost of the strategy $(k,\um)$, as it in fact computes the number of tests in the worst case scenario, this is proved rigorously in Lemma~\ref{error-bound}. 
This result can also be derived using an information-based approach since the least informative case is that in which the infected samples are as uniformly distributed as possible which, in the case of interest here, corresponds to having at most one infected individual per pool at all stages. 

\section{Optimal cost, and comparison with the strategy $(3^k,\dots, 3)$}
\label{feas}
In this section we show that the optimal strategy has cost $O(p\log(1/p))$, and compare this cost with that of $(k_3,\um_{33})$, where $k_3(p)$ given in \eqref{floor2} denotes the optimal number of stages within the family of strategies $(k, \um_{33})$ defined in \eqref{i22}. We have seen in Theorem~\ref{c8} that these strategies are optimal for a wide range of infection probabilities $p \in (0,1)$.

\begin{Theorem}
\label{thm}
Let $p\in [0,1-3^{-1/3}]$, $k_3=k_3(p)$ as in \eqref{floor2}. Then
\begin{align}
\label{t.0}
D_{k_3}(\um_{33},p)\le \frac{3}{\log3}\,p\log(1/p)+6p,
\end{align}
Let $\Dopt(p)$ be the cost of the optimal strategy, \eqref{optn}. Then
\begin{align}
\label{t.01}
\Dopt(p)=O\big(p\log(1/p)\big),
\end{align}
and furthermore
\begin{align}
\label{t.1}
\big|D_{k_3}(\um_{33},p)-\Dopt(p)\big|&\le \Big(\,\frac{3}{\log 3}-e\Big)p\log(1/p)+12p+4p^2\log(1/p)\\[2mm]
&\le 0.013\,p\log(1/p)+12p+O\big(p^2\log(1/p)\big).\nn
\end{align}
\end{Theorem}
\begin{proof}
Since the cost of a strategy is bounded by its linearized cost \eqref{negative} we have
\begin{align}
D_{k_3}(\um_{33},p)&\le \oD_{k_3}\big((3^{k_3},\dots, 3),p\big)\nn\\
&=\frac{1}{3^{k_3}}+3pk_3 \quad\qquad \text{by \eqref{linear++} and the constant ratio $\frac{m_{j-1}}{m_j}=3$.} \label{t.2}
\end{align}
Now
\begin{align}
3^{k_3}\ge \frac{\log 3}{3}\frac{1}{\log1/q}\ge  \frac{\log 3}{3}\frac{1}{p(1+p)}\ge \frac{\log 3}{6}\frac{1}{p}, \label{t.3}
\end{align}
and
\begin{align}
3pk_3&\le 3p\log_3 \Bigl(\frac{1}{\log_3( 1/ q )}\Bigr)=-\frac{3p}{\log 3}\left[\log\log(1/q)-\log\log 3\right] \nn\\[2mm]
&\le \frac{3}{\log3}\,p\log(1/p)+\frac{3\log\log3}{\log 3}\,p. \label{t.4}
\end{align}
Apply bounds \eqref{t.3} and \eqref{t.4} to \eqref{t.2} to obtain
\begin{align}
\label{t.6}
D_{k_3}(\um_{33},p)\le \frac{3}{\log3}\,p\log(1/p)+6p,
\end{align}
which is \eqref{t.0}.

We next derive a lower bound. Let $(k,\um) \in \cal{M}$ be the optimal strategy for $p$, so that $\Dopt(p)=D_k(\um,p)$. By Lemma \ref{error-bound}, we have
\begin{align}
\label{t.7}
D_{k_3}(\um_{33},p)\ge D_k(\um,p) &\ge  \oD_k(\um,p) -  \ell m_1 \log^2 q-\ell k p^2, 
\end{align}
where
\begin{align*}
 &\ell=\max_{2\le i\le k}\frac{m_{i-1}}{m_i}\le 4 \,\,\,\text{ (Proposition \ref{p12})},\quad k \le 1+ \log_3\frac{1}{\log_3(1/q)}\,\,\, \text{ (Proposition \ref{p13})},\\
 &\hspace{0cm} \text{and}\quad m_1\le \frac43 \frac{1}{\log_3(1/q)}.
 \end{align*}
 The bound on $m_1$ follows from $m_2|\log(q)|\le \frac{\log(3)}{3}$ (Lemma \ref{p3}) and hence $m_1\le 4m_2<\frac43 \frac{1}{\log_3(1/q)}$. Replace these bounds in \eqref{t.7} and use that by Lemma \ref{optipool} $\oD_{k}(\um,p)\ge ep\log(1/p)$, to get
\begin{align}
\label{t.8}
D_{k_3}(\um_{33},p)\ge\Dopt(p)&\ge \textstyle{ep\log(1/p)-\frac{16}{3} \frac{\log^2(q)}{\log_3(1/q)}-4p^2\log\big(\frac{1}{\log_3(1/q)}\big)-4p^2}\notag\\[1mm]
&\ge\textstyle{ ep\log(1/p) -\frac{16\log(3)}{3}\,p-\frac{4}{\log3}\,p^2\log(1/p)-\big(\frac{4\log\log(3)}{\log(3)}+4\big)p^2}
\end{align}
Inequalities \eqref{t.6} and \eqref{t.8} imply that $\Dopt(p)=O\big(p\log(1/p)\big)$. Moreover,
\begin{align*}
\textstyle{ ep\log(1/p)-6p-4p^2\log(1/p)-5p^215p^2} \le \Dopt(p) \le D_{k_3}\big((3^{k_3},\dots, 3),p\big)\le \textstyle{\frac{3}{\log 3}\,p\log(1/p)+6p,} 
\end{align*}
and in particular
\begin{align*}
\big|D_{k_3}(\um_{33},p)-\Dopt(p)\big|&\le \Big(\,\frac{3}{\log 3}-e\Big)p\log(1/p)+12p+4p^2\log(1/p)+5p^2\\[2mm]
&\le 0.013\,p\log(1/p)+12p+O\big(p^2\log(1/p)\big),
\end{align*}
the bound in \eqref{t.1}. \qedhere
\end{proof}

\appendix
\section{Appendix}
\subsection{Computation of the variance}
\label{variance}
We compute here $\var\bigl(T_k\bigr)$.
We start with $k+1=3$. 
From \eqref{Tdep} we get
\begin{align}
\var\bigl( T_2\bigr)&=\frac{m^2}{m^2_2}\, q^{m_1}\bigl(1-q^{m_1}\bigr)+ m_2^2\,\frac{m_1}{m_2}\, q^{m_2}\bigl(1-q^{m_2}\bigr) \label{var-3}\\
&\quad+2\,m_2^2\sum_{i\neq j}\cov (1-\sY_{1,i}, 1-\sY_{1,j}) +2\,m_1 \sum_{i=1}^{\frac{m_1}{m_2}} \cov (1-\sY_1,1-\sY_{1,i}).\label{i32}
\end{align}
 The first sum in \eqref{i32} vanishes because $1-\sY_j$ and $1-\sY_k$ are independent if $j\neq k$, while 
 \begin{align}
\cov (1-\sY_1, 1-\sY_{1,i})&=\EE \bigl[(1-\sY_1)(1-\sY_{1,i})\big]-\EE [1-\sY_1]\,\EE [1-\sY_{1,i}]\notag\\
&=\EE [1-\sY_{1,i}]-\bigl(1-q^{m_1}\bigr) \bigl(1-q^{m_2}\bigr)  =\bigl(1-q^{m_2}\bigr)\, q^{m_1}. \label{i322}
\end{align}
where we used \eqref{cancel} in the first identity of \eqref{i322}. Replacing in \eqref{var-3} we get
\begin{align}
\label{Var3}
\var\bigl( T_2\bigr)&=\frac{m_1^2}{m^2_2}\, q^{m_1}\bigl(1-q^{m_1}\bigr)+ m_2m_1\, q^{m_2}\bigl(1-q^{m_2}\bigr)+2\,\frac{m_1^2}{m_2}\,q^{m_1}\bigl(1-q^{m_2}\bigr).
\end{align}
 
   The previous argument can be extended to several stages, as long as each pool size is a multiple of the pool size in the following stage. This is the content of \eqref{vark1} and \eqref{vark2} 
   which we prove next.

\begin{proof}[Proof of \eqref{vark1} and \eqref{vark2}]
From \eqref{Tdepk} we get
\begin{align}
&\var\bigl(T_{k}\bigr)
=\mu^2 \Big\{ \var(1-\sY_1)+\dots+\sum_{i_2=1}^{\frac{m_1}{m_2}}\,\, \sum_{i_3=1}^{\frac{m_2}{m_3}} \dots \sum_{i_{k}=1}^{\frac{m_{k-1}}{m_k}}\var\bigl(1-\sY_{1,i_2,\dots, i_{k}}\bigr)\Big\}  \label{F0}\\
&\, +2\mu^2\Big\{ \sum_{i_2=1}^{\frac{m_1}{m_2}}\cov \bigl(1-\sY_1,1-\sY_{1,i_2}\bigr)+\dots+\sum_{i_2=1}^{\frac{m_1}{m_2}}\,\, \sum_{i_3=1}^{\frac{m_2}{m_3}} \dots \sum_{i_{k}=1}^{\frac{m_{k-1}}{m_k}}\cov \bigl(1-\sY_1, 1-\sY_{1, i_2, \dots, i_{k}}\bigr)\Big\}\label{F1}\\
&\, +2\mu^2\Big\{ \sum_{i_2=1}^{\frac{m_1}{m_2}} \Big[\sum_{i_3=1}^{\frac{m_2}{m_3}}\cov \bigl(1-\sY_{1,i_2},1-\sY_{1,i_2,i_3}\bigr)+\dots+\sum_{i_3=1}^{\frac{m_2}{m_3}}\,\, \dots \sum_{i_{k}=1}^{\frac{m_{k-1}}{m_k}}\cov \bigl(1-\sY_{1,i_2}, 1-\sY_{1, i_2, \dots, i_{k}}\bigr) \Big] \Big\}\label{F2}\\
&\, +\dots \notag\\
&\, + 2\mu^2\Big\{  \sum_{i_2=1}^{\frac{m_1}{m_2}} \dots \sum_{i_{k-1}=1}^{\frac{m_{k-2}}{m_{k-1}}}\Big[\sum_{i_{k}=1}^{\frac{m_{k-1}}{m_k}}\cov \bigl(1-\sY_{1,i_2,\dots, i_{k-1}}, 1-\sY_{1,i_2,\dots,i_{k}}\bigr)\Big]\Big\}.\label{Fk}
\end{align}
The first line \eqref{F0} follows by adding the variances of each of the sums in \eqref{Tdepk}, and using that terms belonging to the same sum are independent, hence that are no covariance terms arising from each of the individual sums. We then compute the covariances between the different sums, and we take advantage of the fact that if $l<n$, then $1-\sY_{1,i_2, \dots, i_l}$ and $1-\sY_{1,j_2\dots j_n}$ are independent unless $i_\ell=j_\ell$ for all $1\le \ell \le l$, and in this case
\begin{align}
\label{covs}
\cov \bigl(1-\sY_{1,i_2\dots, i_l},1-\sY_{1,i_2\dots, i_n}\bigr)=q^{m_{l+1}}\bigl(1-q^{m_{n+1}}\bigr),
\end{align}
by a computation similar to \eqref{i322}. Recall that $1-\sY_{1,i_2,\dots, i_j} \sim$ Bernoulli$(1-q^{m_{j+1}})$, hence
\begin{align}
\label{varians}
\var(1-\sY_{1,i_2,\dots, i_j})=q^{m_{j+1}}\bigl(1-q^{m_{j+1}}\bigr).
\end{align}
Substituting \eqref{covs} and \eqref{varians} in the expression for the variance above, we have
\begin{align}
\var\bigl(T_{k}\bigr)
=&\mu^2 \Big\{ q^{m_1}\bigl(1-q^{m_1}\bigr)+\mu q^{m_2}\bigl(1-q^{m_2}\bigr)+\dots+ \mu^{k-1}q^{m_k}\bigl(1-q^{m_k}\bigr)\Big\}  \notag\\
&\quad+2\mu^2\Big\{ \mu q^{m_1} \bigl(1-q^{m_2}\bigr)+\mu^2 q^{m_1}\bigl(1-q^{m_3}\bigr)\dots+\mu^{k-1} q^{m_1}\bigl(1-q^{m_k}\bigr)\bigr)\Big\}\notag\\
&\quad +2\mu^2\Big\{ \mu^2 q^{m_2}\bigl(1-q^{m_3}\bigr)+\dots+\mu^{k-1} q^{m_2}\bigl(1-q^{m_k}\bigr) \Big\}\label{Fk+1}\\
&\quad +\dots \notag\\
&\quad + 2\mu^2\Big\{ \mu^{k-1} q^{m_{k-1}}\bigl(1-q^{m_k}\bigr)\Big\}.\notag
\end{align}
If we rewrite \eqref{Fk+1} by collecting all terms that have a factor $(1-q^{m_i}),\,1\le i\le k$, we get the expression in \eqref{vark1}.
\end{proof}

\subsection{Error in the linear approximation}
\label{aB}
\begin{proof}[Proof of Lemma  \ref{error-bound}]
We have
\begin{align}
\label{errors}
  D_k(\um,p)-\oD_k(\um,p) &=\Bigl(1-e^{m_k\log q}-m_k p\Bigr)+ \sum_{j=2}^k \frac1{m_j} \Bigl(1-e^{m_{j-1}\log q} -m_{j-1}p\Bigr).
\end{align}
To show that the error is non positive it suffices to prove that 
\begin{align*}
  f(p) = 1-e^{m_j\log q} -m_jp \,\le 0,\qquad \ 0\le p<1
\end{align*}
for any given $1\le j\le k$.
We have $f(0) =0$ and 
\begin{align*}
  f'(p) &= \frac{m_j}{q} e^{m_j \log q}-m_j\\
        &= m_j \Bigl[\frac{e^{m_j\log q}}{q}-1\Bigr] = m_j \Bigl[\frac{ q^{m_{j}}}{q}-1\Bigr]\\
  &= m_j \Bigl[ q^{m_{j}-1}-1\Bigr] \le 0.
\end{align*}
because $m_j\ge 1$. This implies that $f$ is decreasing and $f(p)<0$ for all $p$, and item {\it i)} in the lemma follows.

To prove item {\it ii)}, note that by the inequality $|1-e^x+x|\le \frac{x^2}{2}$ on $x\le 0$, we have
\begin{align*}
\left|\frac{1-e^{m_{j-1}\log q}}{m_j}+\frac{m_{j-1}}{m_j}\log q \right| \le \frac12 \frac{m^2_{j-1}\log^2 q}{m_j}.
\end{align*}
Denote $\um=(m_1,\dots,m_k)$ and recall the notation $m_{k+1}:=1$. Then
\begin{align}
\label{F1}
\Big|D_k(\um,p)-\Bigl(\,\frac{1}{m_1}-m_k\log q-\sum_{j=2}^k \frac{m_{j-1}}{m_j} \log q\Bigr)\Big|&\le \frac12\sum_{j=2}^{k+1} \frac{m^2_{j-1}\log^2 q}{m_j}\nn \\
&\le \frac12\, \ell\, \log^2 q \sum_{j=2}^{k+1} \frac{m_1}{2^{j-2}}\quad \text{ using $m_j\le \frac{m_1}{2^{j-1}}$}\nn\\
&\le \ell m_1 \log^2 q. 
\end{align}
On the other hand
\begin{align}
\label{F2}
\Big|L_k(\um,p)-\Bigl(\,\frac{1}{m_1}-m_k\log q-\sum_{j=2}^k \frac{m_{j-1}}{m_j} \log q\Bigr)\Big|&=\bigr|p+\log q\bigl| \,\sum_{j=2}^{k+1}\frac{m_{j-1}}{m_j}\nn\\
&\le \,k\ell p^2,
\end{align}
where the last line follows from the inequality $|x+\log(1-x)|\le x^2$, $0\le x\le\frac12$. The result follows from \eqref{F1} and \eqref{F2}.

\end{proof}

\subsection{Positive definite Hessian matrix}
\label{hessian}
We prove here that the critical point \eqref{gpools} is indeed a minimum of $\oD_k$, for given $k$.
The Hessian matrix of $\oD_k$ is a tridiagonal symmetric matrix $H_k$ with entries 
\[
\begin{array}
{ll}
\displaystyle{ H_{11}=\frac{2}{m_1^3},\qquad H_{ii}= \frac{2pm_{i-1}}{m_i^3},}\qquad\qquad &2\le i\le k,\\
\displaystyle{H_{i\, i+1}=H_{i+1\, i}=-\frac{p}{m_{i+1}^2},}&1\le i\le k-1,\\
\text{and}\quad \displaystyle{ H_{ij}=0} &\text{if }|i-j|>1.
\end{array}
\]
Let us denote by $H^\sharp=H\big((m_1^\sharp(k),\dots,m_k^\sharp(k)\big)$. To simplify notation, let $\mu:=p^{-\frac{1}{k+1}}$, so that $m_j^\sharp(k)=\mu^{k-j+1}$.
We have 
\[
\begin{array}
{ll}
\displaystyle{ H_{ii}=2\mu p^3\mu^{2i},}\qquad\qquad\quad &2\le i\le k,\\
\displaystyle{H_{i\, i+1}=H_{i+1\, i}}=\displaystyle{- p^3\mu^{2(i+1)},}\qquad\qquad\qquad &1\le i\le k-1,\\
\text{and}\quad \displaystyle{ H_{ij}=0} &\text{if }|i-j|>1.
\end{array}
\]
Given $\mathbf{x}=(x_1,\dots,x_k) \in \R^k$, let us define $\mathbf{y}:=\big(\mu x_1,\mu^2 x_2,\dots,\mu^k x_k\big)$. We compute
\begin{align*}
\mathbf{x}^{\text{t}}H^\sharp\mathbf{x}&=p^3\Big[\sum_{i=1}^k 2\mu \mu^{2i} x_i^2 -2\sum_{i=1}^{k-1} \mu^{2(i+1)} x_i x_{i+1}\Big]\\
&=p^3\big[\sum_{i=1}^k 2\mu y_i^2 -2\sum_{i=1}^{k-1} \mu y_i y_{i+1}\Big]\\
&=\mu p^3\Big[y_1^2+(y_1-y_2)^2+(y_2-y_3)^2+\dots+(y_{k-1}-y_k)^2+y_k^2\Big]\ge 0,
\end{align*}
and $\mathbf{x}^{\text{t}}H^\sharp\mathbf{x}=0$ if and only if $y_1=y_2=\dots y_k=0$, or, in terms of the original vector, $\mathbf{x}=0$. We conclude that $H^\sharp$ is positive definite.

\section*{Acknowledgments} 
We would like to thank Pablo Aguilar, Alejandro Colaneri, Hugo Menzella, Juliana Sesma and Sergio Chialina for bringing this problem to our attention and encouraging us to study it. We thank Leandro Mart\'{\i}nez and Walter Mascarenhas for their help with the accurate computation of $\Phi(p)$ when $p = 2^{-51}$ in Conjecture \ref{p14}. 
P.A.F. would like to thank Luiz-Rafael Santos for comments and reference suggestions.

We thank the referees and the associated editor for their comments and suggestions. 

This work was partially supported by UBA (UBACyT 20020170100482BA, 20020160100155BA) and ANPCyT (PICT
2015-3824, 2015-3583, 2018-02026, 2018-02842).

\bibliographystyle{siam}


\begin{thebibliography}{10}

\bibitem{ahn2021adaptive}
{\sc S.~Ahn, W.-N. Chen, and A.~Ozgur}, {\em Adaptive group testing on networks
  with community structure}, 2021.
\newblock arXiv 2101.02405.

\bibitem{aldridge}
{\sc M.~Aldridge}, {\em Rates of adaptive group testing in the linear regime},
  in 2019 IEEE International Symposium on Information Theory (ISIT), 04 2019,
  pp.~236--240.

\bibitem{review-johnson}
{\sc M.~Aldridge, O.~Johnson, and J.~Scarlett}, {\em Group Testing: An
  Information Theory Perspective}, Foundations and Trends in Communications and
  Information Theory Series, Now Publishers, 2019.

\bibitem{Allemann}
{\sc A.~Allemann}, {\em An efficient algorithm for combinatorial group
  testing}, in Information theory, combinatorics, and search theory, vol.~7777
  of Lecture Notes in Comput. Sci., Springer, Heidelberg, 2013, pp.~569--596.

\bibitem{affmmpd-medrxiv}
{\sc I.~Armend{\'a}riz, P.~A. Ferrari, D.~Fraiman, J.~M. Mart{\'\i}nez, H.~G.
  Menzella, and S.~Ponce~Dawson}, {\em Nested pool testing strategy for the
  diagnosis of infectious diseases}, Scientific Reports, 11 (2021), p.~18108.

\bibitem{pooling-webpage}
{\sc I.~Armend{\'a}riz, P.~A. Ferrari, D.~Fraiman, J.~M. Mart{\'\i}nez, H.~G.
  Menzella, S.~Ponce~Dawson, and F.~N.~C. Sobral}, {\em Pool testing webpage}.
\newblock \url{http://www.pooling.df.uba.ar/}, 2021.
\newblock Accessed: 2021-07-14.

\bibitem{bertolotti2020network}
{\sc P.~Bertolotti and A.~Jadbabaie}, {\em Network group testing}, 2020.
\newblock arXiv 2012.02847.

\bibitem{bilder-tebbs}
{\sc C.~R. Bilder and J.~M. Tebbs}, {\em Pooled-testing procedures for
  screening high volume clinical specimens in heterogeneous populations},
  Statistics in Medicine, 31 (2012), pp.~3261--3268.

\bibitem{bilder2015}
{\sc M.~S. Black, C.~R. Bilder, and J.~M. Tebbs}, {\em Optimal retesting
  configurations for hierarchical group testing}, Journal of the Royal
  Statistical Society: Series C (Applied Statistics), 64 (2015), pp.~693--710.

\bibitem{damas}
{\sc P.~Damaschke and A.~S. Muhammad}, {\em Randomized group testing both
  query-optimal and minimal adaptive}, in S{OFSEM} 2012: theory and practice of
  computer science, vol.~7147 of Lecture Notes in Comput. Sci., Springer,
  Heidelberg, 2012, pp.~214--225.

\bibitem{Dong2020.03.14.20036129}
{\sc L.~Dong, J.~Zhou, C.~Niu, Q.~Wang, Y.~Pan, S.~Sheng, X.~Wang, Y.~Zhang,
  J.~Yang, M.~Liu, Y.~Zhao, X.~Zhang, T.~Zhu, T.~Peng, J.~Xie, Y.~Gao, D.~Wang,
  Y.~Zhao, X.~Dai, and X.~Fang}, {\em Highly accurate and sensitive diagnostic
  detection of sars-cov-2 by digital pcr}, medRxiv,  (2020).

\bibitem{dorfman}
{\sc R.~Dorfman}, {\em The detection of defective members of large
  populations}, Ann. Math. Statist., 14 (1943), pp.~436--440.

\bibitem{finucan}
{\sc H.~M. Finucan}, {\em The blood testing problem}, Journal of the Royal
  Statistical Society. Series C (Applied Statistics), 13 (1964), pp.~43--50.

\bibitem{gabrys2021acdc}
{\sc R.~Gabrys, S.~Pattabiraman, V.~Rana, J.~Ribeiro, M.~Cheraghchi,
  V.~Guruswami, and O.~Milenkovic}, {\em Ac-dc: Amplification curve diagnostics
  for covid-19 group testing}, 2021.
\newblock arXiv 2011.05223.

\bibitem{goenka2020contact}
{\sc R.~Goenka, S.-J. Cao, C.-W. Wong, A.~Rajwade, and D.~Baron}, {\em Contact
  tracing enhances the efficiency of covid-19 group testing}, 2020.
\newblock arXiv 2011.14186.

\bibitem{doi:10.1080/00401706.1972.10488888}
{\sc L.~E. Graff and R.~Roeloffs}, {\em Group testing in the presence of test
  error; an extension of the {D}orfman procedure}, Technometrics, 14 (1972),
  pp.~113--122.

\bibitem{hanel2020boosting}
{\sc R.~Hanel and S.~Thurner}, {\em Boosting test-efficiency by pooled testing
  strategies for sars-cov-2}, 2020.

\bibitem{zbMATH03222508}
{\sc D.~A. {Huffman}}, {\em {A method for the construction of
  minimum-redundancy codes}}.
\newblock {Proc. IRE 40, No. 9, 1098--1101 (1952); Russian translation in
  Kibern. Sb. 3, 79--87 (1961).}, 1952.

\bibitem{hwang}
{\sc F.~K. Hwang}, {\em A method for detecting all defective members in a
  population by group testing}, Journal of the American Statistical
  Association, 67 (1972), pp.~605--608.

\bibitem{MR1214790}
{\sc N.~L. Johnson, S.~Kotz, and X.~Z. Wu}, {\em Inspection errors for
  attributes in quality control}, vol.~44 of Monographs on Statistics and
  Applied Probability, Chapman \& Hall, London, 1991.

\bibitem{kim_2007}
{\sc H.-Y. Kim, M.~G. Hudgens, J.~M. Dreyfuss, D.~J. Westreich, and C.~D.
  Pilcher}, {\em Comparison of group testing algorithms for case identification
  in the presence of test error}, Biometrics, 63 (2007), pp.~1152--1163.

\bibitem{MR673451}
{\sc S.~Kotz and N.~L. Johnson}, {\em Errors in inspection and grading:
  distributional aspects of screening and hierarchal screening}, Comm. Statist.
  A---Theory Methods, 11 (1982), pp.~1997--2016.

\bibitem{Li}
{\sc C.~H. Li}, {\em A sequential method for screening experimental variables},
  J. Amer. Statist. Assoc., 57 (1962), pp.~455--477.

\bibitem{MTB}
{\sc C.~S. McMahan, J.~M. Tebbs, and C.~R. Bilder}, {\em Informative {D}orfman
  screening}, Biometrics, 68 (2012), pp.~287--296.

\bibitem{Mentus2020.04.05.20050245}
{\sc C.~Mentus, M.~Romeo, and C.~DiPaola}, {\em Analysis and applications of
  non-adaptive and adaptive group testing methods for covid-19}, medRxiv,
  (2020).

\bibitem{mezard-toninelli}
{\sc M.~M\'{e}zard and C.~Toninelli}, {\em Group testing with random pools:
  optimal two-stage algorithms}, IEEE Trans. Inform. Theory, 57 (2011),
  pp.~1736--1745.

\bibitem{nikolopoulos2021community}
{\sc P.~Nikolopoulos, T.~Guo, S.~R. Srinivasavaradhan, C.~Fragouli, and
  S.~Diggavi}, {\em Community aware group testing}, 2021.
\newblock arXiv 2007.08111.

\bibitem{nikolopoulos2021group}
{\sc P.~Nikolopoulos, S.~R. Srinivasavaradhan, T.~Guo, C.~Fragouli, and
  S.~Diggavi}, {\em Group testing for overlapping communities}, 2021.
\newblock arXiv 2012.02804.

\bibitem{papanicolaou2020binary}
{\sc V.~G. Papanicolaou}, {\em A binary search scheme for determining all
  contaminated specimens}, 2020.
\newblock ArXiv 2007.11910.

\bibitem{Sinnott-Armstrong2020.03.27.20043968}
{\sc N.~Sinnott-Armstrong, D.~Klein, and B.~Hickey}, {\em Evaluation of group
  testing for sars-cov-2 rna}, medRxiv,  (2020).

\bibitem{sobel1959group}
{\sc M.~Sobel and P.~A. Groll}, {\em Group testing to eliminate efficiently all
  defectives in a binomial sample}, Bell System Technical Journal, 38 (1959),
  pp.~1179--1252.

\bibitem{sobel}
{\sc M.~{Sobel} and P.~A. {Groll}}, {\em Group testing to eliminate efficiently
  all defectives in a binomial sample}, The Bell System Technical Journal, 38
  (1959), pp.~1179--1252.

\bibitem{srinivasavaradhan2021dynamic}
{\sc S.~R. Srinivasavaradhan, P.~Nikolopoulos, C.~Fragouli, and S.~Diggavi},
  {\em Dynamic group testing to control and monitor disease progression in a
  population}, 2021.
\newblock arXiv 2106.10765.

\bibitem{sterrett}
{\sc A.~Sterrett}, {\em On the detection of defective members of large
  populations}, The Annals of Mathematical Statistics, 28 (1957),
  pp.~1033--1036.

\bibitem{Suo2020.02.29.20029439}
{\sc T.~Suo, X.~Liu, M.~Guo, J.~Feng, W.~Hu, Y.~Yang, Q.~Zhang, X.~Wang,
  M.~Sajid, D.~Guo, Z.~Huang, L.~Deng, T.~Chen, F.~Liu, K.~Xu, Y.~Liu,
  Q.~Zhang, Y.~Liu, Y.~Xiong, G.~Guo, Y.~Chen, and K.~Lan}, {\em ddpcr: a more
  sensitive and accurate tool for sars-cov-2 detection in low viral load
  specimens}, medRxiv,  (2020).

\bibitem{ungar1960cutoff}
{\sc P.~Ungar}, {\em The cutoff point for group testing}, Communications on
  Pure and Applied Mathematics, 13 (1960), pp.~49--54.

\bibitem{Yelin2020.03.26.20039438}
{\sc I.~Yelin, N.~Aharony, E.~Shaer-Tamar, A.~Argoetti, E.~Messer,
  D.~Berenbaum, E.~Shafran, A.~Kuzli, N.~Gandali, T.~Hashimshony,
  Y.~Mandel-Gutfreund, M.~Halberthal, Y.~Geffen, M.~Szwarcwort-Cohen, and
  R.~Kishony}, {\em Evaluation of {COVID-19 RT-qPCR} test in multi-sample
  pools}, medRxiv,  (2020).

\bibitem{MR3569134}
{\sc N.~Zaman and N.~Pippenger}, {\em Asymptotic analysis of optimal nested
  group-testing procedures}, Probab. Engrg. Inform. Sci., 30 (2016),
  pp.~547--552.

\end{thebibliography}

\end{document}